\documentclass[11pt,leqno]{imsart}
\usepackage[english]{babel}
\usepackage[margin=3.2cm]{geometry}                 
\RequirePackage{amsmath,amsthm,amsfonts,amssymb}
\RequirePackage[authoryear]{natbib}                 
\RequirePackage[colorlinks,linkcolor=blue,citecolor=blue,urlcolor=blue]{hyperref}
\allowdisplaybreaks[4]
\usepackage[scr=rsfso,cal=boondox,frak=euler,bb=ams]{mathalfa}  
\newcommand{\overbar}[1]{\mkern 1.5mu\overline{\mkern-1.5mu#1\mkern-1.5mu}\mkern 1.5mu}   
\numberwithin{equation}{section}                                
\swapnumbers   
                                                 
\theoremstyle{plain}
\newtheorem{theorem}[equation]{Theorem}
\newtheorem{lemma}[equation]{Lemma}
\newtheorem{proposition}[equation]{Proposition}
\theoremstyle{definition}                                       

\newtheorem{rems}[equation]{Remarks}
\newtheorem{rem57}[equation]{Remark to Proposition \ref{prop5-7}}
\def\thickhline{\noalign{\hrule height.8pt}}

\usepackage{fancyhdr}
\usepackage{xurl}    

\DeclareRobustCommand{\bigO}{%
  \text{\usefont{OMS}{cmsy}{m}{n}O}%
}

\makeatletter
\renewenvironment{proof}[1][\proofname]{\par
\pushQED{\qed}%
\normalfont \topsep6\p@\@plus6\p@\relax
\trivlist
\item\relax
{\bfseries 
#1\@addpunct{:}}\hspace\labelsep\ignorespaces 
}{%
\popQED\endtrivlist\@endpefalse
}

\begin{document}
\renewcommand{\theenumi}{\alph{enumi}}  
\renewcommand{\labelenumi}{(\theenumi)}

\begin{frontmatter}
\title{Edgeworth Expansions for Linear Rank Statistics - Consolidated Version}

\begin{aug}
\author[A]{\fnms{Walter}~\snm{Schneller}\ead[label=e1]{walter.schneller@thws.de}}
\footnote[2]{
This is a consolidated version of the article {''Edgeworth Expansions for Linear Rank Statistics''},
published by \href{https://imstat.org/}{The Institute of Mathematical Statistics} in 
\href{https://imstat.org/aos/}{\textit{The Annals of Statistics}},
\href{https://doi.org/10.1214/aos/1176347258}{1989, 1103$-$1123}.
It contains the corrected definition of the distribution of $\underline{I}$ in Section 3, 
as described in the {''Erratum: Edgeworth Expansions for Linear Rank Statistics''}
(published in \href{https://imstat.org/aos/}{\textit{AOS}},
\href{https://doi.org/10.1214/25-AOS2611}{2026, 1649$-$1653}). 
In addition, some minor changes have been made, as listed in the
\href{https://doi.org/10.1214/25-AOS2611SUPPA}{Supplement A} to the 
\href{https://doi.org/10.1214/25-AOS2611}{Erratum}. The only differences between this paper and  
\href{https://doi.org/10.1214/25-AOS2611SUPPB}{Supplement B} 
to the \href{https://doi.org/10.1214/25-AOS2611}{Erratum} are this changed footnote 
and references to the Erratum and its Supplements.} 
\address[A]{\href{https://www.thws.de/en/}{Technische Hochschule W{\"u}rzburg-Schweinfurt},
\href{https://www.thws.de/en/about-thws/faculties/applied-natural-sciences-and-humanities/}
{Fakult{\"a}t Angewandte Natur- und Geisteswissenschaften},
M{\"u}nzstrasse 12, 97070 W{\"u}rzburg\printead[presep={,\ }]{e1}}
\end{aug}

\begin{abstract}
An Edgeworth expansion of first order is established for general linear rank statistics under the null hypothesis 
with a remainder term that is usually of order $n^{-1}$. Furthermore, corresponding results for the second order 
are formulated, but not proved here. The proof for the first order is based on Stein's method and on an extension 
of the combinatorial method of Bolthausen. It is also shown that conditions of van Zwet imply up to a small 
factor our conditions for the validity of Edgeworth expansions.   Moreover, our proof for the  first order also 
provides us with a result about Edgeworth expansions for smooth functions.
\end{abstract}



\end{frontmatter}

\section{Introduction} \label{sec1}
Let $A = (a_{ij})$ be an $n\,{\times}\,n$-matrix of real numbers ($n \geq 2$). Let
\[\mu_{A} = \sum_{i, j} a_{ij}/n,\hspace{20pt}
\sigma_{\!A}^{2} = \sum_{i, j} \breve{a}_{ij}^{2}/(n-1),\hspace*{20pt}
\hat{a}_{ij} = \breve{a}_{ij}/\sigma_{\!A},\]
where
\[\breve{a}_{ij} = a_{ij} - n^{-1}\sum_{l} a_{il} - n^{-1}\sum_{k} a_{kj} + n^{-2}\sum_{k,l} a_{kl}.\]
Furthermore, let
\[\beta_{\!A} = \sum_{i,j} |\hat{a}_{ij}|^{3}\hspace*{10pt}\text{and}
\hspace*{10pt}\delta_{\!A} = \sum_{i,j} |\hat{a}_{ij}|^{4}.\]
If $\pi$ is uniformly distributed on the set $\mathscr{P}_{n}$ of permutations of $\{ 1,\ldots,n \}$, then
asymptotic normality of the linear rank statistic
\[\mathscr{T}_{\!A} = \Bigl( \sum_{i} a_{i\pi(i)} - \mu_{A} \Bigr)/\sigma_{\!A} = \sum_{i} \hat{a}_{i\pi(i)}\]
has been proved under various conditions by \citet{Hoeffding1951}, \citet{Motoo1957} and
others [see also \citet{Schneller1988}]. Results on the rate of convergence have been
obtained, e.g., by \citet{Does1982} (for the case $a_{ij} = e_{i}d_{j}$), \citet{HoChen1978} and
most successfully by \citet{Bolthausen1984}. He proved the existence of an absolute
constant $K > 0$ such that
\begin{equation}\label{eq1-1}
\sup_{z \in \mathbb{R}} \big| \mathscr{F}_{\!A}(z) - \Phi(z) \big| \leq K\beta_{\!A}/n,
\end{equation}
where $\mathscr{F}_{\!A}$ is the distribution function of $\mathscr{T}_{\!A}$ and $\Phi$ is the standard normal
distribution function.

The purpose of this paper is to establish Edgeworth expansions of first order for $\mathscr{F}_{\!A}$ and
$E(q(\mathscr{T}_{\!A}))$ where $q: \mathbb{R} \rightarrow \mathbb{R}$ is sufficiently smooth. To be more precise,
let us consider for a moment $E(q(\mathscr{T}_{\!A}))$ with $q$ only bounded. In order to establish Edgeworth
expansions for this term we have to assume some smoothness of the function $q$ or of the distribution of 
$\mathscr{T}_{\!A}$.

In Theorem \ref{th2-1} we assume that $q$ has in addition a bounded first and second derivative.
But we impose no smoothness condition on the distribution of $\mathscr{T}_{\!A}$.

In Theorem \ref{th2-4} we consider $\mathscr{F}_{\!A}$ (i.e., the functions $q = 1_{(-\infty,\,z]}$, $z \in \mathbb{R}$) and
assume in principle that the second differences of $\mathscr{F}_{\!A}$ fulfill a boundedness condition. In both
theorems the remainder term is of order $\delta_{\!A}/n$. We note that in very many cases $\delta_{\!A}/n$ is of
order $n^{-1}$.

Unfortunately the condition (\ref{eq2-5}) of Theorem \ref{th2-4} is not very practicable 
though it is very natural and almost necessary [see Remark \ref{rems2-11}(a)]. 
But for the case $a_{ij} = e_{i}d_{j}$ we are able to verify this condition
under the practicable conditions of \citet{vanZwet1982} [see Theorem \ref{th2-12}(a)].

Furthermore, for the case of the distribution function $\mathscr{F}_{\!A}$ we formulate corresponding results for
Edgeworth expansions of second order [see Theorem \ref{th2-7} and Theorem \ref{th2-12}(b)]. 
But we do not prove these results and refer the reader for complete proofs 
to the thesis of the author [see \citet{Schneller1987}, or the English translation \citet{Schneller2025ThesisV2}
of an updated version]. In addition, some remarks
on these proofs may be found at the end of this paper (see Section \ref{sec7}). We note that the result of
Theorem \ref{th2-12}(b) contains the result of \citet{Does1983} [see Remark \ref{rems2-18}(c)] and in the two-sample case
(i.e., $a_{ij} = e_{i}d_{j}$ with $e_{1} = \ldots = e_{m} = 0$, $e_{m+1} = \ldots = e_{n} = 1$) this theorem is
comparable (up to a factor $n^{\epsilon}$) to the results of \citet{BivanZw1978} and \citet{Rob1978}.

Section \ref{sec2} contains our results. In Section \ref{sec3} we introduce our two main methods,
namely the method of \citet{stein1972bound} and an extension of the combinatorial method
of \citet{Bolthausen1984}. Using these two methods we prove the basic equation (\ref{eq3-12}).
This equation gives a kind of Edgeworth expansion for $E(q(\mathscr{T}_{\!A}))$ where $q$ is
differentiable and bounded. From this equation we deduce Theorem
\ref{th2-1} in Section \ref{sec4} and Theorem \ref{th2-4} in Section \ref{sec5}. 
In Section \ref{sec5} we have to replace the functions
$1_{(-\infty,\,z]}$ by convenient smooth functions (see Lemma \ref{lem5-2}). We note that the main
difference between these two sections is, roughly speaking, the different treatment of the 
two terms $|x||q(x + y) - q(x)|$ and $|q'(x+y) - q'(x)|$. In Section \ref{sec4} we simply apply the mean value theorem 
to both terms [see (\ref{eq4-8a})], while in Section \ref{sec5} 
we need a result like (\ref{eq1-1}) (see Proposition \ref{prop5-7}) for the first 
and condition (\ref{eq2-5}) for the second of these two terms (cf. Proof of Lemma \ref{lem5-3}).

The straightforward use of the mean value theorem in (\ref{eq4-8a}) reveals that the result of Theorem \ref{th2-1} 
is surely not optimal. We have not tried to improve it, since the emphasis of this
paper lies more on Edgeworth expansions for $\mathscr{F}_{\!A}$.

In Section \ref{sec6} we establish the condition (\ref{eq2-5}) of Theorem \ref{th2-4} using a result of \citet{vanZwet1982}. 
Under the conditions (\ref{eq2-13})$-$(\ref{eq2-15}), this result gives an estimate of the characteristic
function of $\mathscr{T}_{\!A}$ for arguments $t$ with $c_{1} \log n \leq |t| \leq c_{2} n^{3/2}$ 
[cf. also (\ref{eq6-6})]. 
From this estimate we deduce (\ref{eq2-5}) essentially with the help of Lemma \ref{lem6-3}.

\section{The results} \label{sec2}
First, we need some notation. Given $F : \mathbb{R} \rightarrow \mathbb{R}$ and $y \in \mathbb{R}$, we define
\begin{equation*}
||F|| = \sup \{|F(z)| : z \in \mathbb{R}\}
\end{equation*}
and the second difference of $F$ related to $y$ by
\begin{equation*}
\Delta^{2}_{y} F(z) = [ F(z+2y) - F(z+y)] - [ F(z+y) - F(z) ],\ \ \ \ z \in \mathbb{R}.
\end{equation*}
The interpolating polynomial to $F$ of degree 2 at the points $z$, $z+y$ and $z + 2y$ is
\begin{equation*}
\begin{array}{l@{\hspace*{0.8ex}}c@{\hspace*{0.8ex}}l}
P^{2}_{y}(x;z,F)&=&F(z) + [F(z+y) - F(z)](x-z)y^{-1}\\[1.5ex]
&&+\, \Delta^{2}_{y} F(z)\frac{1}{2}(x - z - y)(x - z)y^{-2},\ \ \ \ x, z \in \mathbb{R}.
\end{array}
\end{equation*}
For the matrix $A$ and $l \in \mathbb{N}_{0}$ we define $\hat{A} = ( \hat{a}_{ij} )$ and
\begin{equation*}
\begin{array}{r@{\hspace*{0.8ex}}c@{\hspace*{0.8ex}}l}
F_{\!A}&=&\text{distribution function of } \displaystyle{T_{\!A} = \sum_{i} a_{i\pi(i)}},\\[2.1ex]
M(l,A)&=&\text{set of all $(n-l)\,{\times}\,(n-l)$ matrices, which can be obtained}\\[0.8ex]
&&\text{from $A$ by cancelling $l$ rows and $l$ columns,}\\[1.7ex]
N(l,A)&=&\bigcup \big\{ M(r,A): 0 \leq r \leq \min\{l, n - 1\}\big\},\\[1.5ex]
D_{\!A}&=&(\delta_{\!A}/n)^{1/2},\ \ \ \  
E_{\!A} = \displaystyle{\Bigl( \sum\nolimits_{i,j} |\hat{a}_{ij}|^{5}/n \Bigr)^{1/3}}.
\end{array}
\end{equation*}
We notice that using $\displaystyle{\sum\nolimits_{i,j} \hat{a}_{ij}^{2} = n-1}$ and H{\"o}lder's inequality we get
\begin{equation*}
\beta_{A}/n \leq D_{\!A} \leq E_{\!A}
\ \ \ \ \text{and}\ \ \ \
\sum\nolimits_{i,j} |\hat{a}_{ij}|^{k}/n \geq 2^{-(k/2)} n^{1-(k/2)}\ \  \text{for}\ k > 2
\end{equation*}
[cf. \citet{Schneller2025ThesisV2}, Lemma 3.1.18(a), (b)]. Finally, the expansions are
\[\begin{array}{r@{\hspace*{0.8ex}}c@{\hspace*{0.8ex}}l@{\hspace*{0.8ex}}c@{\hspace*{0.8ex}}l}
e_{1,A}(x)&=&\Phi(x) - \psi(x)\, \dfrac{1}{6}\, \lambda_{1,A}\,(x^2-1),\\[2ex]
e_{2,A}(x)&=&\Phi(x) - \psi(x)\, \biggl\{\, \dfrac{1}{6}\, \lambda_{1,A}\,(x^2-1)&+&   
\dfrac{1}{24}\, \lambda_{2,A}\,(x^3-3x)\\[2ex]
&&&+&\dfrac{1}{72}\,\lambda_{1,A}^{2}\,(x^5-10x^3 + 15x)\,\biggr\},
\end{array}
\]
\begin{tabular}{@{}l@{\hspace*{3.8ex}}l@{\hspace*{2.8ex}}l@{}}
where&$\psi = \Phi'$,&
$\displaystyle{\lambda_{1,A} = n^{-1} \sum_{i,j} \hat{a}^{3}_{ij}}$\hspace*{4ex}and\\[3.2ex]
&&$\displaystyle{\lambda_{2,A} = n^{-1} \sum_{i,j} \hat{a}_{ij}^{4} + 3n^{-1} -  
3n^{-2} \sum_{i,j,k}\, \Bigl( \hat{a}_{ij}^{2} \hat{a}_{ik}^{2} + 
\hat{a}_{ij}^{2} \hat{a}_{kj}^{2} \Bigr)}$.
\end{tabular}\\[0.2ex]

Here is our first result:
\begin{theorem}\label{th2-1}
There exist positive absolute constants $K_{1}$, $K_{2}$ and $K_{3}$ such that 
for all $A$ satisfying $\sigma_{\!A} > 0$ and all
\begin{equation} \label{eq2-2}
q \in \mathscr{D} = \bigl\{ g : \mathbb{R} \rightarrow \mathbb{R} : 
g\, \text{is}\,\, \text{twice}\,\, \text{differentiable}\,\, \text{and}\,\, 
g,\, g', g''\, \text{are}\,\, \text{bounded}\, \bigr\},
\end{equation}
we have
\begin{equation}\label{eq2-3}
\bigg| E(q(\mathscr{T}_{\!A})) - \int\nolimits_{\mathbb{R}} q e'_{1,A}\, dx \bigg| \leq 
\bigl( K_{1}||q|| + K_{2}||q'|| +  K_{3}||q''|| \bigr) D^{2}_{\!A}.\vspace*{0.5ex}
\end{equation}
\end{theorem}

The next result deals with Edgeworth expansions of first order for the distribution function $\mathscr{F}_{\!A}$.
\begin{theorem}\label{th2-4}
There exist positive absolute constants $K_{4}$ and $K_{5}$ such that 
for all $A$ satisfying $\sigma_{\!A} > 0$ and the condition
\begin{equation} \label{eq2-5}
\begin{array}{c}
\text{there}\,\, \text{exists}\,\, \text{a}\,\, \text{positive}\,\, \text{constant}\,\, C_{1}\,\,
\text{such}\,\, \text{that}\\[1.5ex]
\big|\Delta^{2}_{y} F_{B}(z)\big| \leq C_{1}\bigl( D^{2}_{\!A} + y^{2} \bigr)\\[1.5ex]
\text{for}\,\, \text{all}\,\, z \in \mathbb{R},\, 0 \leq y \leq D_{\!A}\,\, \text{and}\,\, B \in N(8, \hat{A}),
\end{array}
\end{equation}
we have
\begin{equation} \label{eq2-6}
\big|\big| \mathscr{F}_{\!A} - e_{1,A} \big|\big| \leq 
\bigl( K_{4}C_{1} + K_{5} \bigr) D^{2}_{\!A}.\vspace*{0.5ex}
\end{equation}
\end{theorem}

The corresponding result of this theorem for the second order is
\begin{theorem}\label{th2-7}
There exist positive absolute constants $K_{6}$, $K_{7}$ and $K_{8}$ such that 
for all $A$ satisfying $\sigma_{\!A} > 0$ and the conditions
\begin{equation} \label{eq2-8}
\begin{array}{c}
\text{there}\,\, \text{exists}\,\, \text{a}\,\, \text{positive}\,\, \text{constant}\,\, C_{2}\,\,
\text{such}\,\, \text{that}\\[1.5ex]
\big| F_{B}(x) - P^{2}_{E_{\!A}}(x;z,F_{B}) \big| \leq C_{2}\bigl( E^{3}_{\!A} + (x-z)^{3} \bigr)\\[1.5ex]
\text{for}\,\, \text{all}\,\, z \in \mathbb{R},\, z \leq x \leq z + 3 E_{\!A}\,\, \text{and}\,\, B \in N(16, \hat{A}),
\end{array}
\end{equation}
\begin{equation} \label{eq2-9}
\begin{array}{c}
\text{there}\,\, \text{exists}\,\, \text{a}\,\, \text{positive}\,\, \text{constant}\,\, C_{3}\,\,
\text{such}\,\, \text{that}\\[1.5ex]
\bigl( |z| + 1 \bigr) \big| \Delta^{2}_{y} F_{B}(z) \big| \leq C_{3}\bigl( E^{2}_{\!A} + y^{2} \bigr)\\[1.5ex]
\text{for}\,\, \text{all}\,\, z \in \mathbb{R},\, 0 \leq y \leq E_{\!A}\,\, \text{and}\,\, B \in N(16, \hat{A}),
\end{array}
\end{equation}
we have
\begin{equation} \label{eq2-10}
\big|\big| \mathscr{F}_{\!A} - e_{2,A} \big|\big| \leq 
\bigl( K_{6}C_{2} + K_{7}C_{3} + K_{8} \bigr) E^{3}_{\!A}.\vspace*{0.5ex}
\end{equation}
\end{theorem}

\begin{rems}\label{rems2-11}
(a) The conditions (\ref{eq2-5}), (\ref{eq2-8}) and (\ref{eq2-9}) are analogous to those necessary and sufficient
conditions which \citet{BiRob1982} used to establish Edgeworth expansions for the i.i.d. case. 

Following their arguments on page 502, one can deduce from (\ref{eq2-10}) the conditions (\ref{eq2-8}) 
and (\ref{eq2-9}) for $B = \hat{A}$ and
from (\ref{eq2-6}) the condition (\ref{eq2-5}) for $B = \hat{A}$ 
(with new $C_{1}$, $C_{2}$, $C_{3}$) [cf. \citet{Schneller2025ThesisV2}, Proposition 3.1.20,
for a complete proof].

However, note that these arguments do not show (\ref{eq2-8}) and (\ref{eq2-9})
[(\ref{eq2-5})] for general $B \in N(16, \hat{A})$
[$B \in N(8, \hat{A})$].

(b) It seems likely, at least to the present author, that the Theorems \ref{th2-4} and \ref{th2-7} remain correct, if we
assume (\ref{eq2-5}), (\ref{eq2-8}) and (\ref{eq2-9}) only for $\hat{A}$ instead of $B \in N(8, \hat{A})$
[$B \in N(16, \hat{A})$]. However, a proof eludes me.
\end{rems}

In the special case where $a_{ij} = e_{i}d_{j}$, we can replace (\ref{eq2-5}), (\ref{eq2-8}) and (\ref{eq2-9})
by the conditions of \citet{vanZwet1982}. If we define
\[\bar{e} = \sum_{i} e_{i}/n,\hspace{20pt}
\bar{d} = \sum_{j} d_{j}/n,\hspace{20pt}
x^{+} = \max \{ x, 0 \}\ \ \ \text{for}\,\, x \in \mathbb{R}
\]
and write $\lambda$ for the Lebesgue measure, we have
\begin{theorem}\label{th2-12}
Suppose that there exist positive constants $e$, $E$, $d$, $D$ and $\delta$ such that
\begin{equation} \label{eq2-13}
\begin{array}{@{}l@{\hspace*{4ex}}l@{}}
\displaystyle{\sum_{i} \big| e_{i} - \bar{e} \big|^{r} \geq e n,}&
\displaystyle{\sum_{i} \big| e_{i} - \bar{e} \big|^{k} \leq E n}\\[1.5ex]
&\hspace*{13.3ex}\text{for}\,\, \text{some}\,\, k > 2\,\, \text{and}\,\,\, 0 < r < k,
\end{array}
\end{equation}
\begin{equation} \label{eq2-14}
\begin{array}{@{}l@{\hspace*{4ex}}l@{}}
\displaystyle{\sum_{j} \big| d_{j} - \bar{d} \big|^{m} \geq d n,}&
\displaystyle{\sum_{j} \big| d_{j} - \bar{d} \big|^{s} \leq D n}\\[1.5ex]
&\hspace*{13.3ex}\text{for}\,\, \text{some}\,\, s > 2\,\, \text{and}\,\,\, 0 < m < s,
\end{array}
\end{equation}
\begin{equation} \label{eq2-15}
\begin{array}{@{}l@{}}
\displaystyle{\lambda\bigl(\bigl\{ x : \big| x - d_{j} \big| < \zeta\,\,\text{for}\,\, \text{some}\,\, 1 \leq j \leq n
\bigr\}\bigr) \geq \delta n \zeta}\\[1.5ex]
\hspace*{43ex}\text{for}\,\, \text{some}\,\, \zeta \geq n^{-3/2}\,\log n.
\end{array}
\end{equation}
(a) Then there exist positive constants $\mathscr{K}_{1}$ and $\mathscr{K}_{2}$ depending only on 
$e$, $E$, $d$, $D$, $\delta$ and $r$, $k$, $m$, $s$ such that
\begin{equation} \label{eq2-16}
\begin{array}{@{}l@{\hspace*{0.8ex}}c@{\hspace*{0.8ex}}l@{}}
\big|\big| \mathscr{F}_{\!A} - e_{1,A} \big|\big|&\leq&\mathscr{K}_{1} (\log n)^2 D^{2}_{\!A}\\[1.5ex] 
&\leq&\mathscr{K}_{2} (\log n)^2 n^{- 1\, +\, ((4/k)\, -\, 1)^{+}\, 
+\, ((4/s)\, -\, 1)^{+}}.
\end{array}
\end{equation}
(b) Let $\epsilon > 0$. Then there exist positive constants $\mathscr{K}_{3}$ and $\mathscr{K}_{4}$ depending only on 
$e$, $E$, $d$, $D$, $\delta$, $r$, $k$, $m$, $s$ and $\epsilon$ such that
\begin{equation} \label{eq2-17}
\begin{array}{@{}l@{\hspace*{0.8ex}}c@{\hspace*{0.8ex}}l@{}}
\big|\big| \mathscr{F}_{\!A} - e_{2,A} \big|\big|&\leq&\mathscr{K}_{3} n^{\epsilon} E^{3}_{\!A}\\[1.5ex] 
&\leq&\mathscr{K}_{4} n^{- (3/2)\, +\, \epsilon\, +\, ((5/k)\, -\, 1)^{+}\, 
+\, ((5/s)\, -\, 1)^{+}}.
\end{array}\vspace*{0.5ex}
\end{equation}
\end{theorem}

\begin{rems}\label{rems2-18}
(a) For $[((5/k)\, -\, 1)^{+}\, -\, ((4/k)\, -\, 1)^{+}]\, +\, [((5/s)\, -\, 1)^{+}\, -$
\linebreak
$((4/s)\, -\, 1)^{+}]\, <\, \frac{1}{2}$ we can
deduce the second estimate of $|| \mathscr{F}_{\!A} - e_{1,A} ||$ in (\ref{eq2-16}) from
\linebreak
(\ref{eq2-17}). In this case the factor $(\log n)^2$ is superfluous.

(b) Let
\[
d_{j} = E\bigl(J(U_{j:n})\bigr),\hspace*{4ex}j = 1,\ldots,n\ \text{(exact scores)},
\]
where $J : (0,1) \rightarrow \mathbb{R}$ is an integrable function and $U_{j:n}$ denotes the $j$th order
statistic in a random sample of size $n$ from the uniform distribution on $(0,1)$.

If $J$ is nonconstant, continuously differentiable and satisfies $\int_{0}^{1} |J(t)|^{s} dt < \infty$
for some $s > 2$, then $n_{e}(J) = \sup \{ m \in \mathbb{N} : \sum_{1}^{m} | d_{j} - \bar{d}| = 0 \} < \infty$,
and (\ref{eq2-14}) (with this $s$!) and (\ref{eq2-15}) are fulfilled for all $n > n_{e}(J)$ with constants
$d$, $D$ and $\delta$ depending only on $J$ and $s$. For a proof see \citet{Schneller2025ThesisV2}, 
proof of Theorem 4.4.35(a), (b).

(c) Let
\[
d_{j} = J\Bigl(\frac{j}{n+1}\Bigr),\hspace*{4ex}j = 1,\ldots,n\ \text{(approximate scores)},
\]
where $J : (0,1) \rightarrow \mathbb{R}$ is a function.

If $J$ is nonconstant, continuously differentiable and satisfies
\begin{equation} \label{eq2-19}
\begin{array}{@{}l@{\hspace*{0.8ex}}c@{\hspace*{0.8ex}}l@{}}
|J'(t)|&\leq&\Gamma \bigl( t(1-t) \bigr)^{- 1\, -\, (1/s)\, +\, \beta}\hspace*{3ex}
\text{for}\,\, \text{all}\,\, t \in (0,1),\\[1.5ex]
&&\hspace*{15ex}\text{for}\,\, \text{some}\,\, \Gamma > 0,\, \mathbb{N} \ni s \geq 3\,\, 
\text{and}\,\,\, 0 < \beta < \dfrac{1}{s},
\end{array}
\end{equation}
then $n_{a}(J) = \sup \{ m \in \mathbb{N} : \sum_{1}^{m} | d_{j} - \bar{d}| = 0 \} < \infty$, and
(\ref{eq2-14}) (with this $s$!) and (\ref{eq2-15}) are fulfilled for all $n > n_{a}(J)$ with constants
$d$, $D$ and $\delta$ depending only on $J$, $s$ and $\beta$. For a proof see \citet{Schneller2025ThesisV2}, 
proof of Theorem 4.4.21(a), (b).

Part (d) of this theorem of \citet{Schneller2025ThesisV2} contains an extension of \citet{Does1983}, Theorem 2.1. 
Does uses expansions which are slightly different from ours. Roughly speaking, he uses 
in his expansions integrals of $J$ whereas we use the corresponding Riemann-sums. 
It is shown in \citet{Schneller2025ThesisV2}, proof of Theorem 4.4.21(d) that, 
if we assume in addition $k \geq 4$ in (\ref{eq2-13}) and $s \geq 4$ in (\ref{eq2-19}), then the
difference between these two expansions is $\bigO(n^{- (3/2)\, +\, 3((1/s)\, -\, \beta)})$.
Combining this with the rate of (\ref{eq2-17}), we obtain a result with a better convergence rate 
and with weaker assumptions for $J$ and $e_{i}$ than Does. 
He obtained the rate ${\scriptstyle \bigO}(n^{-1})$ and assumed especially for $J$ that
\[
\limsup\limits_{t \rightarrow\, 0,1} t(1-t)|J''(t)/J'(t)| < 2
\hspace*{3ex}\text{and}\hspace*{3ex}
|J'''(t)| \leq \overline{\Gamma}\bigl(t(1-t)\bigr)^{- 3\, -\, (1/14)\, +\, \beta}
\]
with $\overline{\Gamma}$, $t$ and $\beta$ as in (\ref{eq2-19}). But from the last inequality, (\ref{eq2-19})
follows with $s = 14$.
\end{rems}

Finally, we remark that the constants introduced in this section remain fixed throughout the paper.
In contrast $c, c_{1}, c_{2}, \ldots$ denote positive constants which depend only on the formula
where they appear.

\section{Proof of the basic equation}  \label{sec3}
In this and the next (the next but one) section we prove Theorem \ref{th2-1} (Theorem \ref{th2-4}). For that
we fix an $n\,{\times}\,n$-matrix $A$ with $\sigma_{\!A} > 0$. Of course, we may assume $a_{ij} = \hat{a}_{ij}$.
Furthermore, let $q : \mathbb{R} \rightarrow \mathbb{R}$ be a function which is assumed to be bounded and
differentiable throughout this section.

The essence of Stein's method is that if we define
\begin{equation} \label{eq3-1}
f(x) = (\Theta q)(x) = \psi(x)^{-1} \int_{- \infty}^{x} \bigl( q(y) - \Phi(q) \bigr) \psi(y) dy
\end{equation}
[$\Phi(q)$ denotes the standard normal expectation of $q$], then we obtain the differential equation
\begin{equation} \label{eq3-2}
f'(x) - xf(x) = q(x) - \Phi(q),\ \ \ \ \ x \in \mathbb{R},
\end{equation}
and thus for any random variable $\xi$ we have
\begin{equation} \label{eq3-3}
E(q(\xi)) - \Phi(q) = E(f'(\xi)) - E(\xi f(\xi)).
\end{equation}
Therefore, in order to estimate $E(q(\xi)) - \Phi(q)$, we can estimate the expression $E(f'(\xi)) - E(\xi f(\xi))$.

Assume for a moment that $\xi = S_{n} = n^{-1/2} (X_{1} + \ldots + X_{n})$ where $X_{1},\ldots,$ $X_{n}$
are i.i.d with $E(X_{i}) = 0$, $V(X_{i}) = 1$. Then we have
\begin{equation} \label{eq3-4}
E(f'(S_{n})) - E(S_{n} f(S_{n})) = E(f'(S_{n})) - \sqrt{n}E(X_{n} f(S_{n})).
\end{equation}
Define $S_{n-1}^{n} = n^{-1/2} (X_{1} + \ldots + X_{n-1})$. Now, in order to prove the classical CLT
or Berry-Esseen theorem, we have to carry out a Taylor expansion of first order of $f$ about $S_{n-1}^{n}$ in
$E(X_{n} f(S_{n-1}^{n} + n^{-1/2} X_{n}))$ and then apply the independence of
$S_{n-1}^{n}$ and $X_{n}$ [for more details see, e.g., \citet{Bolthausen1984}, Section 2].
For Edgeworth expansions we have to use Taylor expansions of higher order of $f$ and $f'$ [for more details
see, e.g., \citet{Schneller2025ThesisV2}, Chapter 1, especially Section 3 or the related paper of \citet{Barbour1986}].

However, in our case we have $\xi = T_{\!A}$ and therefore there is a priori no comparable independence. For that reason
we use an extension of Bolthausen's combinatorial method. This method yields a ''bit'', but for us enough
independence.

For an intuitive understanding of this method, 
it should be noted first that the {''source''} of the independence of this method is the
independence of the random vector $\underline{I} = (I_{1},\ldots,I_{8})$ and the random permutation $\pi_{1}$
(see text below after equation (\ref{eq3-4d})). This
independence remains valid between $\underline{I}$ and the random
permutations $\pi_{2}$, $\pi_{3}$ and $\pi_{4}$, which are derived from $\pi_{1}$ (see Lemma \ref{lem3-5}(a)).

Secondly, now it is the best to inspect Table \ref{tab2} from right to left and
consider the results of Lemma \ref{lem3-5}(b)$-$(d) as well as 
(\ref{eq3-5a})$-$(\ref{eq3-8}). 

In order to prove our analogue of (\ref{eq3-4}),
we use the independence of $\pi_{4}$ and $I_{1}$. This is done in (\ref{eq3-14}) where we obtain
$E(T_{\!A} f(T_{\!A})) = n E(a_{I_{1}J_{1}} f(T_{4}))$.

For the next step we need the independence of $a_{I_{1}J_{1}}$ and a part of $T_{4}$.  
$(I_{1}, J_{1})$ and $\pi_{4}$ are \textit{not} independent, but $(I_{1}, J_{1})$ and $\pi_{3}$ are,
which is achieved through an ''exchange'' of $J_{1}$ and $J_{2}$. Thus $a_{I_{1}J_{1}}$ and the
part $T_{3}$ of $T_{4} = T_{3} + (T_{4} - T_{3})$ are independent. 

Next, we need the independence  of $T_{4} - T_{3} = a_{I_{1}J_{1}} + a_{I_{2}J_{2}} -
a_{I_{1}J_{2}} - a_{I_{2}J_{1}}$ and a part of $T_{3}$. $(I_{1}, I_{2}, J_{1}, J_{2})$ and $\pi_{3}$
are \textit{not} independent, but $(I_{1}, I_{2}, J_{1}, J_{2})$ and $\pi_{2}$ are, which is achieved
through an ''exchange'' of the blocks $(J_{1}, J_{2})$ and $(J_{3}, J_{4})$. Thus $T_{4} - T_{3}$ and
the part $T_{2}$ of $T_{3} = T_{2} + (T_{3} - T_{2})$ are independent.

Now it is clear that we find our last independence statement through an ''exchange'' of the blocks
$(J_{1},\ldots,J_{4})$ and $(J_{5},\ldots,J_{8})$. We get, that $T_{3} - T_{2}$ and the
part $T_{1}$ of $T_{2} = T_{1} + (T_{2} - T_{1})$ are independent.

After these considerations we give the explicit construction which starts with the last step of the above
considerations and ends with the first step of these considerations. Let
\[
\begin{array}{@{}l@{\hspace*{0.8ex}}c@{\hspace*{0.8ex}}
l@{\hspace*{0.4ex}}l@{\hspace*{1.6ex}}l@{\hspace*{1.6ex}}
l@{\hspace*{1.6ex}}l@{\hspace*{1.6ex}}l@{}}
N&=&\multicolumn{5}{@{\hspace*{0.2ex}}l@{}}{\bigl\{ 1,\ldots,n \bigr\},}\\[1ex]
M_{4}&=&\bigl\{&\multicolumn{5}{@{\hspace*{0.2ex}}l@{}}
{\overline{i} = ( i_{1},\ldots,i_{4} ) \in N^{4}\,:\,\overline{i}\,\, \text{satisfies}\,\, \text{the}\,\,
\text{equivalence}:}\\[1ex]
&&&\text{(t1)}&i_{1} = i_{2}&\Leftrightarrow&i_{3} = i_{4}
\,\bigr\},\\[1ex]
M_{8}&=&\bigl\{&\multicolumn{5}{@{\hspace*{0.2ex}}l@{}}
{\underline{i} = ( i_{1},\ldots,i_{8} ) \in N^{8}\,:\,\underline{i}\,\, \text{satisfies}\,\, \text{the}\,\,
\text{equivalences}:}\\[1ex]
&&&\text{(t1)}&i_{1} = i_{2}&\Leftrightarrow&i_{3} = i_{4};\\[1.2ex]
&&&\text{(u1)}&i_{1} = i_{2}&\Leftrightarrow&i_{7} = i_{8};\\[1.2ex]
&&&\text{(u2)}&i_{3} = i_{4}&\Leftrightarrow&i_{5} = i_{6};\\[1.2ex]
&&&\text{(u3)}&i_{l} = i_{k}&\Leftrightarrow&i_{l+6} = i_{k+2}&\text{for}\ l = 1, 2;\ \, k = 3,4
\,\bigr\}.
\end{array}
\]

For each $\underline{i} \in M_{8}$ we fix once and for all permutations $u(\underline{i})$, $t(\underline{i})$
and $s(\underline{i})$ of $N$ with properties as described in Table \ref{tab1}.

\begin{table*}[h!]
\centering
\caption{Definition of the permutations $u(\underline{i})$, $t(\underline{i})$
and $s(\underline{i})$}
\label{tab1}
\begin{tabular}{@{}c|c|c|c@{}} \thickhline
\rule[-1.4ex]{0ex}{4.4ex}\hspace*{11.2ex}&
\hspace*{6.3ex}{\boldmath $u(\underline{i})$}\hspace*{6.3ex}&
\hspace*{6.3ex}{\boldmath $t(\underline{i})$}\hspace*{6.3ex}&
\hspace*{6.3ex}{\boldmath $s(\underline{i})$}\hspace*{6.3ex}\\ \thickhline
\rule[-0.8ex]{0ex}{3ex}$i_{1}$&$i_{7}$&$i_{4}$&$i_{2}$\\
\rule[-1ex]{0ex}{3.2ex}$i_{2}$&$i_{8}$&$i_{3}$&$i_{1}$\\ \hline
\rule[-1ex]{0ex}{3.6ex}$i_{3}$&$i_{5}$&Values&\\
\rule[-1.2ex]{0ex}{3.6ex}$i_{4}$&$i_{6}$&$\in \{\,i_{1},\ldots,i_{4}\,\}$&\\ \cline{1-3}
\rule[-0.8ex]{0ex}{3.2ex}$i_{5}$&&\multicolumn{2}{|c}{}\\
\rule[-0.8ex]{0ex}{3ex}$i_{6}$&Values&\multicolumn{2}{|c}{Remain}\\
\rule[-0.8ex]{0ex}{3ex}$i_{7}$&$\in \{\,i_{1},\ldots,i_{8}\,\}$&\multicolumn{2}{|c}{fixed}\\
\rule[-0.8ex]{0ex}{3ex}$i_{8}$&&\multicolumn{2}{|c}{}\\ \cline{1-2}
\rule[-1ex]{0ex}{3.6ex}$N \setminus \{\,i_{1},\ldots,i_{8}\,\}$&\multicolumn{3}{|c}{}\\ \thickhline
\end{tabular}
\end{table*}

We remark that the map $\underline{i} \rightarrow u(\underline{i})$ is well defined because of
(u1), (u2), (u3) and that the map $\underline{i} \rightarrow t(\underline{i})$ is well defined because of (t1).

As the reader will see, we have not defined all values of $u(\underline{i})$ and $t(\underline{i})$ explicitly.
The reason for that is that we do not need an explicit definition and that we would have to consider a lot of cases
for the explicit definition of these values (e.g. $i_{4} = i_{5}$, then $[u(\underline{i})](i_{5}) = i_{6}$;
but $i_{4} \not= i_{5}$, then $[u(\underline{i})](i_{5}) \not= i_{6}$).

For each $\underline{i} \in M_{8}$ other useful notations are
\begin{align} 
\label{eq3-4z1}
\gamma(\underline{i}) &= |\{i_{1}, i_{2}\}| = |\{i_{3}, i_{4}\}|,\\[1ex]
\label{eq3-4z2}
\theta(\underline{i}) &= |\{i_{1}, i_{2}, i_{3}, i_{4}\}| = |\{i_{5}, i_{6}, i_{7}, i_{8}\}|.
\end{align}
The second {''=''} in (\ref{eq3-4z1}) is valid due to (t1).
Similarly, the second {''=''} in (\ref{eq3-4z2}) is from (u1), (u2) and (u3).

Furthermore, let $\underline{I} = (I_{1},\ldots,I_{8})$ be a random element on $M_{8}$ whose 
distribution can be described as follows:
\begin{equation} \label{eq3-4a}
(I_{1},I_{2})\,\, \text{is uniformly distributed on}\,\, N^{2}.
\end{equation}
\begin{equation} \label{eq3-4b}
\begin{array}{c}
\text{For each}\,\, (i_{1},i_{2}) \in N^{2},\\[0.8ex]
(I_{3},\,I_{4})\,\Big|\,(I_{1},\,I_{2}) = (i_{1},\,i_{2})\\[1.5ex]
\text{is uniformly distributed on}\,\, 
\bigl\{(\eta_{3},\eta_{4}) : ( i_{1},i_{2},\eta_{3},\eta_{4} ) \in M_{4} \bigr\}.
\end{array}
\end{equation} 
\begin{equation} \label{eq3-4c}
\begin{array}{c}
\text{For each}\,\, (i_{1},i_{2},i_{3},i_{4}) \in M_{4},\\[0.8ex] 
(I_{5}, I_{6}, I_{7}, I_{8})\,\Big|\,(I_{1},\,I_{2},\,I_{3},\,I_{4}) = (i_{1},\,i_{2},\,i_{3},\,i_{4})\\[1.5ex]
\text{is uniformly distributed on}\\[1ex] 
\bigl\{(\eta_{5},\eta_{6},\eta_{7},\eta_{8}) :
( i_{1},i_{2},i_{3},i_{4},\eta_{5},\eta_{6},\eta_{7},\eta_{8}) \in M_{8} \bigr\}.
\end{array}
\end{equation}\vspace*{0.3ex}

To summarize, the multiplication rule together with (\ref{eq3-4z1}) and (\ref{eq3-4z2}) yields 
\begin{equation} \label{eq3-4d}
P\bigl( \underline{I} = \underline{i} \bigr) =
\dfrac{\bigl( n - \theta(\underline{i}) \bigr)!}{n!} \cdot
\dfrac{\bigl( n - \gamma(\underline{i}) \bigr)!}{n!} \cdot \dfrac{1}{n^2}\,\,\,\,\,\,\,\,\,\,\,
\text{for each}\,\,\underline{i} \in M_{8}.
\end{equation}

Additionally, let $\pi_{1}$ be a random permutation that is uniformly distributed on $\mathscr{P}_{n}$
and independent of $\underline{I}$. Furthermore, we define
\[
\begin{array}{l}
\pi_{2} = \pi_{1} \circ u(\underline{I}),\hspace*{2.8ex}
\pi_{3} = \pi_{2} \circ t(\underline{I}),\hspace*{2.8ex}
\pi_{4} = \pi_{3} \circ s(\underline{I}),\\[1.5ex]
J_{1} = \pi_{1}(I_{5}),\hspace*{2.8ex}
J_{2} = \pi_{1}(I_{6}),\hspace*{2.8ex}
J_{3} = \pi_{1}(I_{7}),\hspace*{2.8ex}
J_{4} = \pi_{1}(I_{8}),\\[1.5ex]
J_{5} = \pi_{1}(I_{1}),\hspace*{2.8ex}
J_{6} = \pi_{1}(I_{2}),\hspace*{2.8ex}
J_{7} = \pi_{1}(I_{3}),\hspace*{2.8ex}
J_{8} = \pi_{1}(I_{4}).
\end{array}
\]

Using the definition of $u$, $t$ and $s$, we see in Table \ref{tab2} how $\pi_{1},\ldots,\pi_{4}$
map $I_{1},\ldots,I_{8}$. 

\begin{table*}[h!]
\centering
\caption{Values of $I_{1},\ldots,I_{8}$ under $\pi_{1}, \pi_{2}, \pi_{3}, \pi_{4}$}
\label{tab2}
\begin{tabular}{@{}c|c|c|c|c@{}} \thickhline
\rule[-1.4ex]{0ex}{4.4ex}\hspace*{11.2ex}&
\hspace*{4.5ex}{\boldmath $\pi_{1}$}\hspace*{4.5ex}&
\hspace*{4.5ex}{\boldmath $\pi_{2}$}\hspace*{4.5ex}&
\hspace*{4.5ex}{\boldmath $\pi_{3}$}\hspace*{4.5ex}&
\hspace*{4.5ex}{\boldmath $\pi_{4}$}\hspace*{4.5ex}\\ \thickhline
\rule[-0.8ex]{0ex}{3ex}$I_{1}$&$J_{5}$&$J_{3}$&$J_{2}$&$J_{1}$\\
\rule[-0.8ex]{0ex}{3ex}$I_{2}$&$J_{6}$&$J_{4}$&$J_{1}$&$J_{2}$\\ \cline{3-5}
\rule[-0.8ex]{0ex}{3ex}&&&Random&\\
\rule[-0.8ex]{0ex}{3ex}$I_{3}$&$J_{7}$&$J_{1}$&variables&Same\\
\rule[-0.8ex]{0ex}{3ex}$I_{4}$&$J_{8}$&$J_{2}$&$\in \sigma(I_{1},\ldots,I_{4},$&bei $\pi_{3}$\\ 
\rule[-0.8ex]{0ex}{3ex}&&&$\ \ \ \ \,\,\,\,J_{1},\ldots,J_{4})$&\\ \cline{2-5}
\rule[-0.8ex]{0ex}{3ex}$I_{5}$&$J_{1}$&Random&&\\
\rule[-0.8ex]{0ex}{3ex}$I_{6}$&$J_{2}$&variables&Same&Same\\
\rule[-0.8ex]{0ex}{3ex}$I_{7}$&$J_{3}$&$\in \sigma(I_{1},\ldots,I_{8},$&as $\pi_{2}$&as $\pi_{2}$\\
\rule[-0.8ex]{0ex}{3ex}$I_{8}$&$J_{4}$&$\ \ \ \ \,\,\,\,J_{1},\ldots,J_{8})$&&\\ \thickhline
\end{tabular}
\end{table*}

As usual we define that $\sigma(X)$ is the $\sigma$-field generated by the
random vector $X$ and $f \in \sigma(X)$ means that $f$ is measurable relative to $\sigma(X)$.

The most important statements regarding distribution and independence 
for the above construction are contained in the following lemma.

\begin{lemma}\label{lem3-5}\ \nopagebreak\\[1.5ex]
\nopagebreak
(a) $\pi_{1}$, $\pi_{2}$, $\pi_{3}$ and $\pi_{4}$ are uniformly distributed on $\mathscr{P}_{n}$
and independent of $\underline{I}$.\\[1.5ex]
(b) $\pi_{1}$ and $(I_{1},\ldots,I_{4},J_{1},\ldots,J_{4})$ are independent.\\[1.5ex]
(c) $\pi_{2}$ and $(I_{1},I_{2},J_{1},J_{2})$ are independent.\\[1.5ex]
(d) $\pi_{3}$ and $(I_{1},J_{1})$ are independent.\\[1.5ex]
(e) $(I_{l},\pi_{k}(I_{l}))$ is uniformly distributed on $N^2$ for all $1 \leq l \leq 8$, $1  \leq k \leq 4$.\\[-1.5ex]
\end{lemma}
\begin{proof}
(a) For all $\pi \in \mathscr{P}_{n}$ and $\underline{i} \in M_{8}$ we have
\[P(\pi_{2} = \pi,\, \underline{I} = \underline{i}) = 
P(\pi_{1} = \pi \circ u(\underline{i})^{-1},\, \underline{I} = \underline{i}) =
\frac{1}{n!} P(\underline{I} = \underline{i}).
\]
Thus summation over all $\underline{i} \in M_{8}$ gives first the law of $\pi_{2}$ and then the
independence of $\pi_{2}$ and $\underline{I}$. The assertion for $\pi_{3}$ and $\pi_{4}$ follows
analogously.

(b) For all $\pi \in \mathscr{P}_{n}$ and $\underline{i} \in M_{8}$ we have
\[\begin{array}{l}
P\bigl( (I_{1}, \ldots, I_{4}, J_{1}, \ldots, J_{4}) = \underline{i},\,\pi_{1} = \pi \bigr)\\[1.2ex]
\ \ \ =\ P\Bigl( \underline{I} = \bigl(\,i_{1}, \ldots, i_{4}, \pi^{-1}(i_{5}), \ldots, \pi^{-1}(i_{8})\bigr)\Bigr)\,
P( \pi_{1} = \pi )\\[2.3ex]
\ \ \ =\ P( \underline{I} = \underline{i} ) P( \pi_{1} = \pi ).
\end{array}
\]
The last equation is valid due to (\ref{eq3-4c}).
Now summation over all $\pi \in \mathscr{P}_{n}$ gives that $(I_{1},\ldots,I_{4},J_{1},\ldots,J_{4})$ 
and $\underline{I}$ have the same law,
from which the assertion follows.

(c) Proceed as in (b) using $(I_{1},I_{2},J_{1},J_{2})$ and $\pi_{2}$ instead of 
$(I_{1},\ldots,I_{4},J_{1},\ldots,J_{4})$ and $\pi_{1}$. Furthermore, utilize
$J_{1} = \pi_{2}(I_{3})$, $J_{2} = \pi_{2}(I_{4})$ from Table \ref{tab2}
and (\ref{eq3-4b}) instead of (\ref{eq3-4c}).

(d) Proceed as in (b) using $(I_{1},J_{1})$ and $\pi_{3}$ instead of 
$(I_{1},\ldots,I_{4},J_{1},\ldots,J_{4})$ and $\pi_{1}$. Furthermore, utilize
$J_{1} = \pi_{3}(I_{2})$ from Table \ref{tab2}
and (\ref{eq3-4a}) instead of (\ref{eq3-4c}).

(e) From $\underline{i} \in M_{8} \Leftrightarrow \underline{i} + (1,1,\ldots,1) \in M_{8}\ (\text{mod}\ n)$ we conclude
that $I_{1},\ldots,I_{8}$ are uniformly distributed over $N$. Thus (e) follows easily using (a). 
\end{proof}

Next, we define
\begin{align}
\label{eq3-5a}
T_{k} &= \displaystyle{\sum_{j} a_{j\pi_{k}(j)}}\ \ \ \ \ \text{for}\,\, k = 1, 2, 3, 4,\\
\label{eq3-5b}
\Delta T_{k} &= T_{k+1} - T_{k}\ \ \ \ \ \text{for}\,\, k = 1, 2, 3.
\end{align}
From Table \ref{tab2} and the definition of $u(\underline{i})$, $t(\underline{i})$ and $s(\underline{i})$
we obtain
\begin{align}
&\Delta T_{1} = \displaystyle{\sum_{l = 1}^{8} \bigl( a_{I_{l}\pi_{2}(I_{l})} -
a_{I_{l}\pi_{1}(I_{l})} \bigr) \in \sigma( I_{1},\ldots,I_{8},J_{1},\ldots,J_{8} )},\label{eq3-6}\\
&\Delta T_{2} = \displaystyle{\sum_{l = 1}^{4} \bigl( a_{I_{l}\pi_{3}(I_{l})} -
a_{I_{l}\pi_{2}(I_{l})} \bigr) \in \sigma( I_{1},\ldots,I_{4},J_{1},\ldots,J_{4} )},\label{eq3-7}\\[0.5ex]
&\Delta T_{3} = \displaystyle{a_{I_{1}J_{1}} + a_{I_{2}J_{2}} - a_{I_{1}J_{2}} - a_{I_{2}J_{1}}
\in \sigma( I_{1},I_{2},J_{1},J_{2} )}.\label{eq3-8}
\end{align} 
Furthermore, a simple calculation gives
\begin{equation}
E( a_{I_{1}J_{1}} ) = 0,\ \ \ \ \ \ n E\bigl( a_{I_{1}J_{1}} \Delta T_{3} \bigr) = 1,\label{eq3-9}
\end{equation}
\begin{equation}
n \Bigl\{ E\bigl( a_{I_{1}J_{1}}\,\Delta T_{3}\,\Delta T_{2} \bigr)\, +\,
\frac{1}{2} E\bigl( a_{I_{1}J_{1}} \bigl(\Delta T_{3})^2 \bigr) \Bigr\} 
= \frac{1}{2} E\bigl(T_{\!A}^3\bigr).\label{eq3-10}
\end{equation}

Now we are able to prove the following basic equation.
\begin{lemma}\label{lem3-11}
Let $f$ and $q$ be connected as in (\ref{eq3-1}). Then
\begin{equation} \label{eq3-12}
E\bigl( q(T_{\!A}) \bigr) - \Phi( q )
+ \frac{1}{2} E\bigl( T_{\!A}^3 \bigr)\,E\bigl( T_{\!A}\,f'(T_{\!A}) \bigr) = R(q),
\end{equation}
where
\begin{align*}
R(q) = &\displaystyle{\dfrac{1}{2} E\bigl( T_{\!A}^3 \bigr) n
E\biggl( a_{I_{1}J_{1}} \Delta T_{3} \int\nolimits_{0}^{1}
\bigl( f''(T_{2} + \Delta T_{2} + t\,\Delta T_{3}) - f''(T_{2}) \bigr)\,dt \biggr)}\\[0.7ex]
&\displaystyle{-\ n
E\biggl( a_{I_{1}J_{1}} \Delta T_{3} \Delta T_{2} \int\nolimits_{0}^{1}
\bigl( f''(T_{1} + \Delta T_{1} + t\,\Delta T_{2}) - f''(T_{1}) \bigr)\,dt \biggr)}\\[0.7ex]
&\displaystyle{-\ n
E\biggl( a_{I_{1}J_{1}} ( \Delta T_{3} )^2 \int\nolimits_{0}^{1} (1-t)
\bigl( f''(T_{2} + \Delta T_{2} + t\,\Delta T_{3}) - f''(T_{2}) \bigr)\,dt \biggr)}.
\end{align*}
\end{lemma}

\begin{proof}
The equation (\ref{eq3-3}) with $\xi = T_{\!A}$ gives
\begin{equation} \label{eq3-13}
E\bigl(q(T_{\!A})\bigr) - \Phi(q) = E\bigl(f'(T_{\!A})\bigr) - E\bigl(T_{\!A} f(T_{\!A})\bigr).
\end{equation}
Now, using $J_{1} = \pi_{4}(I_{1})$ from Table \ref{tab2} and the independence of $\pi_{4}$ and $I_{1}$ we obtain
\begin{equation} \label{eq3-14}
E\bigl(T_{\!A} f(T_{\!A})\bigr) = E\bigl(T_{4} f(T_{4})\bigr) = n E\Bigl( a_{I_{1}J_{1}} f(T_{4}) \Bigr).
\end{equation}
A Taylor expansion of $f$ about $T_{3}$ yields further
\begin{align*}
=\; &n E\Bigl( a_{I_{1}J_{1}} f(T_{3}) \Bigr) + n E\Bigl( a_{I_{1}J_{1}} \Delta T_{3} f'(T_{3}) \Bigr)\\[0.5ex]
&+ \displaystyle{n
E\biggl( a_{I_{1}J_{1}} ( \Delta T_{3} )^2 \int\nolimits_{0}^{1} (1-t)
\bigl( f''(T_{2} + \Delta T_{2} + t\,\Delta T_{3}) - f''(T_{2}) \bigr)\,dt \biggr)}\\[1ex]
&+ \dfrac{n}{2} E\Bigl( a_{I_{1}J_{1}} (\Delta T_{3})^2 f''(T_{2}) \Bigr).
\end{align*}
The first summand is zero [cf. Lemma \ref{lem3-5}(d), (\ref{eq3-9})] and the last summand gives
[cf. Lemma \ref{lem3-5}(c), (\ref{eq3-8})]
\begin{equation*}
\dfrac{n}{2} E\Bigl( a_{I_{1}J_{1}} (\Delta T_{3})^2 f''(T_{2}) \Bigr) =
\dfrac{n}{2} E\Bigl( a_{I_{1}J_{1}} (\Delta T_{3})^2 \Bigr) E\bigl( f''(T_{2}) \bigr).
\end{equation*}
For the second summand we have
\begin{alignat*}{3}
&n &&E\Bigl( &&a_{I_{1}J_{1}} \Delta T_{3} f'(T_{3}) \Bigr)\\[0.7ex]
&\ &&=\ &&\!n E\Bigl( a_{I_{1}J_{1}} \Delta T_{3} f'(T_{2}) \Bigr)
+ n E\Bigl( a_{I_{1}J_{1}} \Delta T_{3} \Delta T_{2} f''(T_{1}) \Bigr)\\[0.5ex]
&\ &&\ &&\!\displaystyle{+\ n
E\biggl( a_{I_{1}J_{1}} \Delta T_{3} \Delta T_{2} \int\nolimits_{0}^{1}
\bigl( f''(T_{1} + \Delta T_{1} + t\,\Delta T_{2}) - f''(T_{1}) \bigr)\,dt \biggr)}\\[0.5ex]
&\ &&= &&\!E\bigl(f'(T_{2})\bigr) + n E\Bigl( a_{I_{1}J_{1}} \Delta T_{3} \Delta T_{2} \Bigr) E\bigl( f''(T_{1}) \bigr)
+ \text{last summand}
\end{alignat*}
[cf. Lemma \ref{lem3-5}(b), (c) and (\ref{eq3-7})$-$(\ref{eq3-9})].

Thus, using (\ref{eq3-10}) we conclude
\begin{equation} 
\begin{aligned}  \label{eq3-15}
E\bigl(T_{\!A} f(T_{\!A})\bigr) = &E\bigl(f'(T_{\!A})\bigr) + 
\dfrac{1}{2} E\bigl( T_{\!A}^3 \bigr)\,E\bigl(f''(T_{\!A})\bigr)\\[0.5ex]
&+\ \text{second term of $R(q)$} + \text{third term of $R(q)$}.
\end{aligned}
\end{equation}
The degree of differentiation of $f$ in the term $E\bigl(f''(T_{\!A})\bigr)$ will now be reduced by the
following consideration:
\begin{alignat*}{3}
&E\bigl(T_{\!A} f'(T_{\!A})\bigr) &&=\ &&n E\Bigl( a_{I_{1}J_{1}} f'(T_{4}) \Bigr)\\[0.5ex]
&\ &&= &&n E\Bigl( a_{I_{1}J_{1}} f'(T_{3}) \Bigr) +  n E\Bigl( a_{I_{1}J_{1}} \Delta T_{3} f''(T_{2}) \Bigr)\\[0.5ex]
&\ &&\ &&\displaystyle{+\ n
E\biggl( a_{I_{1}J_{1}} \Delta T_{3} \int\nolimits_{0}^{1}
\bigl( f''(T_{2} + \Delta T_{2} + t\,\Delta T_{3}) - f''(T_{2}) \bigr)\,dt \biggr)}\\[1ex]
&\ &&= &&E\bigl(f''(T_{2})\bigr)\\[0.5ex]
&\ &&\ &&\displaystyle{+\ n
E\biggl( a_{I_{1}J_{1}} \Delta T_{3} \int\nolimits_{0}^{1}
\bigl( f''(T_{2} + \Delta T_{2} + t\,\Delta T_{3}) - f''(T_{2}) \bigr)\,dt \biggr)}.
\end{alignat*}
Implementing this in (\ref{eq3-15}) and using (\ref{eq3-13}) we obtain the lemma. 
\end{proof}

\section{Proof of Theorem \ref{th2-1}} \label{sec4}
In addition to the assumptions of the last section ($a_{ij} = \hat{a}_{ij}$ for all $i$, $j$ and $q$ is bounded,
differentiable), we have to make some further assumptions.

First, we assume that $q \in \mathscr{D}$; see (\ref{eq2-2}). Furthermore, we may assume 
$\beta_{\!A} \leq \epsilon_{0} n$ und $n \geq n_{0}$ for arbitrary but fixed $0 < \epsilon_{0} \leq 1$
and $n_{0} \geq 10$. These constants will be specified later in Lemma \ref{lem4-1}. 
For $\beta_{\!A} > \epsilon_{0} n$ or $n < n_{0}$ we obtain (\ref{eq2-3}) from
\[
\displaystyle{\bigg| E(q(\mathscr{T}_{\!A})) - \int\nolimits_{\mathbb{R}} q e'_{1,A}\, dx \bigg| \leq ||q|| \bigl( 2 + \beta_{\!A}/n \bigr)}
\]
and $\beta_{\!A}/n \leq D_{\!A}$, $0 < c \leq \beta_{\!A}$, if we take $K_{1}$ large enough.

Moreover, we must assume $|a_{ij}| \leq 1$ for all $i$, $j$. We do not show here how this truncation is established
and refer the reader to \citet{Schneller2025ThesisV2}, Chapter 3, Section 4. For the basic ideas of this truncation
one may also consult \citet{Bolthausen1984}, pages 382 and 383. We mention that for this truncation
we eventually have to reduce the above $\epsilon_{0}$. 

We devide the following proof of Theorem 2.1 in two parts. The first part is 
\begin{lemma}\label{lem4-1}
There exist positive constants $c_{1}$, $c_{2}$ and $c_{3}$ such that
\begin{equation*}
|R(q)| \leq \bigl( c_{1}||q|| + c_{2}||q'|| +  c_{3}||q''|| \bigr) D^{2}_{\!A}.
\end{equation*}
\end{lemma}

\begin{proof}
Let $\overline{I} = (I_{1},\ldots,I_{4})$ and $\overline{i}$, $\underline{I}$, $\underline{i}$ 
as in Section \ref{sec3}. Furthermore, let
$\overline{J}$, $\overline{j}$, $\underline{J}$, $\underline{j}$ be defined analogously to
$\overline{I}$, $\overline{i}$, $\underline{I}$, $\underline{i}$.

Because of Lemma \ref{lem3-5}(e) and (\ref{eq3-6})$-$(\ref{eq3-8}) we have for $k,l,m \in \{1, 2, 3\}$
\begin{align*}
n E\Bigl( |a_{I_{1}J_{1}}|\,|\Delta T_{k}|\,|\Delta T_{l}| \Bigr) &\leq c_{4} \beta_{A}/n\ \ \ \text{and}\\[0.5ex]  
n E\Bigl( |a_{I_{1}J_{1}}|\,|\Delta T_{k}|\,|\Delta T_{l}|\,|\Delta T_{m}| \Bigr) &\leq c_{5} D_{\!A}^{2},
\end{align*}
so that it remains to show that there exist
constants $c_{6},\ldots,c_{11}$ such that
\begin{equation} \label{eq4-2}
\begin{array}{@{}l@{}}
\displaystyle{\Big| E\Bigl( f''\bigl(\,T_{1} + \Delta T_{1} + t\,\Delta T_{2}\,\bigr)
- f''\bigl(\,T_{1}\,\bigr)\, \big|\,
\underline{I} = \underline{i},\, \underline{J} = \underline{j} \Bigr) \Big|}\\[1.5ex]
\hspace*{4.2ex}\displaystyle{\leq\ \bigl( c_{6}||q|| + c_{7}||q'|| +  c_{8}||q''|| \bigr) 
E\Bigl( | \Delta T_{1} | + | \Delta T_{2} |
\,\big|\,\underline{I} = \underline{i},\, \underline{J} = \underline{j} \Bigr)}\\[2ex]
\hspace*{4.2ex}\text{for}\,\, \text{all}\,\, 0 \leq t \leq 1\,\,\, \text{and}\,\,\, \underline{i}, 
\underline{j} \in M_{8}\,\,\,
\text{with}\,\,\, P\bigl( \underline{I} = \underline{i},\, \underline{J} = \underline{j} \bigr) > 0,
\end{array}
\end{equation}
\begin{equation} \label{eq4-3}
\begin{array}{@{}l@{}}
\displaystyle{\Big| E\Bigl( f''\bigl(\,T_{2} + \Delta T_{2} + t\,\Delta T_{3}\,\bigr)
- f''\bigl(\,T_{2}\,\bigr)\, \big|\,
\overline{I} = \overline{i},\, \overline{J} = \overline{j} \Bigr) \Big|}\\[1.5ex]
\hspace*{4.2ex}\displaystyle{\leq\ \bigl( c_{9}||q|| + c_{10}||q'|| +  c_{11}||q''|| \bigr) 
E\Bigl( | \Delta T_{2} | + | \Delta T_{3} |
\,\big|\,\overline{I} = \overline{i},\, \overline{J} = \overline{j} \Bigr)}\\[2ex]
\hspace*{4.2ex}\text{for}\,\, \text{all}\,\, 0 \leq t \leq 1\,\,\, \text{and}\,\,\, \overline{i}, 
\overline{j} \in M_{4}\,\,\,
\text{with}\,\,\, P\bigl( \overline{I} = \overline{i},\, \overline{J} = \overline{j} \bigr) > 0.
\end{array}
\end{equation}
Since the proofs of (\ref{eq4-2}) and (\ref{eq4-3}) are very similar, we only show 
(\ref{eq4-2}). 
For that we fix the quantities $\underline{i}$, $\underline{j}$ and $t$.

Now we look at the conditional distribution of $T_{1}$ given $\underline{I} = \underline{i}$,
$\underline{J} = \underline{j}$. $T_{1}$ depends only on $\pi_{1}$ and the conditional distribution of $\pi_{1}$
is easy to describe: $\pi_{1}$ takes any permutation $\varphi$ which satisfies $\varphi(i_{k}) = j_{k+4}$
for $1 \leq k \leq 4$ and $\varphi(i_{k}) = j_{k-4}$ for $5 \leq k \leq 8$ with equal probability.
Therefore, $T_{1}$ given $\underline{I} = \underline{i}$, $\underline{J} = \underline{j}$ has the same
law as
\[
\displaystyle{\sum\limits_{(i,\,j)\, \in\, S} a_{ij} + T_{B}},
\]
where
\[
S = \bigl\{ ( i_{k}, j_{k+4} ):\,1 \leq k \leq 4 \bigr\}\,
\cup\, \bigl\{ ( i_{k}, j_{k-4} ):\,5 \leq k \leq 8 \bigr\}
\]
and $B$ is the $(n-l)\,{\times}\,(n-l)$-matrix which is obtained from $A$ by cancelling the rows
$i_{1},\ldots,i_{8}$ and the columns $j_{1},\ldots,j_{8}$. Using the notations
\[\begin{array}{c}
a = \displaystyle{\sum\limits_{(i,\,j)\, \in\, S} a_{ij}},\ \ \ \ \ \ 
p = E\bigl( \Delta T_{1}\,\big|\,\underline{I} = \underline{i},\, \underline{J} = \underline{j} \bigr),\\[2.5ex]
r = E\bigl( \Delta T_{2}\,\big|\,\underline{I} = \underline{i},\, \underline{J} = \underline{j} \bigr)
\end{array}
\]
[note $\Delta T_{1}, \Delta T_{2} \in \sigma(\underline{I}, \underline{J})$!\,], it remains to show
\begin{equation} \label{eq4-4}
\begin{array}{@{}l@{}}
\displaystyle{\Big| E\Bigl( f''\bigl(\,T_{B} + a + p + tr\,\bigr)
- f''\bigl(\,T_{B} + a\,\bigr) \Bigr) \Big|}\\[2.5ex]
\hspace*{4.2ex}\displaystyle{\leq\ \bigl( c_{12}||q|| + c_{13}||q'|| + c_{14}||q''|| \bigr) 
\bigl( |p| + |r| \bigr)}.
\end{array}
\end{equation}
In order to prove (\ref{eq4-4}) we need the estimates
\begin{equation} \label{eq4-5}
||f|| \leq 4||q||\ \ \ \ \text{and}\ \ \ \ ||f'|| \leq 4||q||
\end{equation}
[\citet{Erickson1974}, Proposition 2.1]. Moreover, differentiation of (\ref{eq3-2}) gives
\begin{equation} \label{eq4-6}
f''(x) = f(x) + xf'(x) + q'(x),\ \ \ \ \ x \in \mathbb{R}.
\end{equation}
From (\ref{eq4-6}) and (\ref{eq4-5}) we obtain
\begin{equation} \label{eq4-7}
\begin{array}{c@{\hspace*{0.8ex}}l}
&\big| f''(x+y) - f''(x) \big|\\[2ex]
\leq&|y|8||q|| + |x|\big| f'(x+y) - f'(x) \big| + \big| q'(x+y) - q'(x) \big|.
\end{array}
\end{equation} 
Using (\ref{eq3-2}), (\ref{eq4-5}) and $|x| \leq 1 + x^2$ we estimate further
\begin{align}
\label{eq4-8}
\,\ &\leq |y|12||q||\{1 + x^2\} + |x|\big| q(x+y) - q(x) \big| + \big| q'(x+y) - q'(x) \big|\\[1.7ex]
\label{eq4-8a}
\,\ &\leq |y|\bigl( 12||q|| + ||q'|| + ||q''|| \bigr)\{1 + x^2\}.
\end{align} 
Therefore we have
\[
\begin{array}{@{}l@{}}
\displaystyle{\Big| E\Bigl( f''\bigl(\,T_{B} + a + p + tr\,\bigr)
- f''\bigl(\,T_{B} + a\,\bigr) \Bigr) \Big|}\\[2.5ex]
\hspace*{4.2ex}\displaystyle{\leq\ \bigl( 12||q|| + ||q'|| + ||q''|| \bigr) 
\bigl( |p| + |r| \bigr) \Bigl\{ 1 + E\bigl( (T_{B} + a)^2  \bigr) \Bigr\}}.
\end{array}
\]
It remains to estimate $E\bigl( (T_{B} + a)^2  \bigr)$. If we take the $0 < \epsilon_{0} \leq 1$
in $\beta_{\!A} \leq \epsilon_{0} n$ small enough and the $n_{0}$ in $n \geq n_{0}$ great enough, we can deduce 
as in \citet{Bolthausen1984}, page 385 after (3.11),
\begin{equation} \label{eq4-9}
|\mu_{B}| \leq 1,\ \ \ \ \ \frac{1}{2} \leq \sigma_{\!B}^{2} \leq \frac{3}{2}
\ \ \ \ \ \text{and}\ \ \ \ \ \beta_{B} \leq c \beta_{\!A}\ \ \ \ \text{for}\,\, B \in N(l,A).
\end{equation}
[The last inequality is needed in (\ref{eq5-4})]. Using $|a_{ij}| \leq 1$ we find further $|a| \leq 8$.
Therefore we obtain \phantom\qedhere  
\begin{equation} \label{eq4-10}              
E\bigl( (T_{B} + a)^2  \bigr) \leq 2 E\bigl( T_{B}^{2} \bigr) + 2 a^{2} = 
2 \bigl( \sigma_{\!B}^{2} + \mu_{B}^{2} \bigr) + 2 a^{2} \leq c.\rlap{$\qquad \hspace*{1.9ex} \Box$}
\end{equation}
\end{proof}

In the second part of this section we prove
\begin{lemma}\label{lem4-11}
There exist positive constants $c_{1}$ and $c_{2}$ such that
\[
\displaystyle{\bigg|\, \dfrac{1}{2} E\bigl( T_{\!A}^{3} \bigr) E\bigl( T_{\!A} f'(T_{\!A}) \bigr) -
\dfrac{1}{2} \lambda_{1,A} \Phi\bigl( x f'(x) \bigr) \bigg|
\leq \bigl( c_{1}||q|| + c_{2}||q'|| \bigr) D^{2}_{\!A}.
}
\]
\end{lemma}
We remark that an easy calculation using (\ref{eq3-2}) shows
\[
\displaystyle{\Phi\bigl( x f'(x) \bigr) = \dfrac{1}{3} \int\nolimits_{\mathbb{R}} q(x) (3x - x^3) \psi(x) dx,}
\]
so that on the left-hand side of the basic equation (\ref{eq3-12}) we have
$E\bigl(q(T_{\!A})\bigr) - \int\nolimits_{\mathbb{R}} q e_{1,A}' dx$ up to a term smaller than a constant times $D_{\!A}^{2}$.

\begin{proof}[Proof of Lemma \ref{lem4-11}]
We have
\[
\begin{array}{@{}l@{}}
\displaystyle{\bigg|\, E\bigl( T_{\!A}^{3} \bigr) E\bigl( T_{\!A} f'(T_{\!A}) \bigr) -
\lambda_{1,A} \Phi\bigl( x f'(x) \bigr) \bigg|}\\[3.5ex]
\hspace*{4.2ex}\displaystyle{\leq E\bigl( \big|T_{\!A} f'(T_{\!A})\big| \bigr) \bigg| E\bigl(T_{\!A}^{3}\bigr) 
- \lambda_{1,A} \bigg|
+ \dfrac{1}{n} \beta_{\!A} \Big| E\bigl( T_{\!A} f'(T_{\!A}) \bigr) - \Phi\bigl( x f'(x) \bigr) \Big|}\\[3ex]
\hspace*{4.2ex}= A_{1} + A_{2}.
\end{array}
\]
To estimate $A_{1}$, we note that for $n \geq 10$
\begin{align*} 
E\bigl( T_{\!A}^{3} \bigr) = &\displaystyle{\dfrac{n}{(n-1)(n-2)} \sum_{i,j} a_{ij}^{3}},\\[0.5ex] 
\text{whereby}\ \dfrac{1}{n} \leq &\dfrac{n}{(n-1)(n-2)} \leq \dfrac{1}{n} + \dfrac{4}{n^2}
\end{align*}
(see \citet{hajek+sidak+sen1999}, Chapter 3, page 90, problem 27).
Using this and (\ref{eq4-5}) we obtain
\begin{equation*}
A_{1} \leq c_{3} ||q|| \beta_{\!A}/n^2 \leq c_{4} ||q|| \beta_{\!A}^{2}/n^2 \leq c_{4} ||q|| D_{\!A}^{2}.
\end{equation*}
For the estimation of $A_{2}$ we show the inequality
\begin{equation} \label{eq4-12}
\Big| E\bigl( T_{\!A} f'(T_{\!A}) \bigr) - \Phi\bigl( x f'(x) \bigr) \Big|
\leq \bigl( c_{5}||q|| + c_{6}||q'|| \bigr) \beta_{\!A}/n.
\end{equation}
In order to prove (\ref{eq4-12}) we define $\mathcal{f}(x) = (\Theta \mathcal{q})(x)$
where $\mathcal{q}(x) = xf'(x)$.
Proceeding as in \citet{Erickson1974} and using (\ref{eq4-5}) we find 
$||\mathcal{f}|| \leq 3 ||f'|| \leq 12||q||$.
From this and (\ref{eq3-2}) for $\mathcal{f}$ we conclude further
$|\mathcal{f}'(x)| \leq |x|16||q|| + 4||q||$ for all $x \in \mathbb{R}$. Thus, using the estimate
\[
\big| (x+y)f'(x+y) - xf'(x) \big| \leq |y| \bigl( 8||q|| + ||q'|| \bigr)\{1 + x^2\},
\ \ \ \ x, y \in \mathbb{R}
\]
[cf. proof of (\ref{eq4-7}), (\ref{eq4-8}) and (\ref{eq4-8a})], and again (\ref{eq3-2}) for $\mathcal{f}$, we have
\begin{equation} \label{eq4-13}
\big| \mathcal{f}'(x+y) - \mathcal{f}'(x) \big| \leq |y| \bigl( c_{7}||q|| + ||q'|| \bigr)
\bigl( 1 + x^2 \bigr) \bigl( 1 + |y| \bigr).
\end{equation}
Now we obtain (\ref{eq4-12}) if we proceed first of all as in \citet{Bolthausen1984}, page 383 bottom and
page 384 [with (\ref{eq4-13}) instead of Bolthausen's (2.5)] and then argue as in our proof of 
Lemma \ref{lem4-1}. Note that Bolthausen's $|\Delta T_{1}|$ and $|\Delta T_{2}|$ are bounded since we have 
$|a_{ij}| \leq 1$ for all $i, j$.
\end{proof}

\section{Proof of Theorem \ref{th2-4}} \label{sec5}
As in the last section we may assume $a_{ij} = \hat{a}_{ij}$ and
$|a_{ij}| \leq 1$ for all $i$, $j$; $\beta_{\!A} \leq \epsilon_{0} n$ und $n \geq n_{0}$ for 
fixed $0 < \epsilon_{0} \leq 1$ and $n_{0} \geq 10$.

In order to apply the results of Section \ref{sec3}, we must replace the discontinuous functions
$1_{(-\infty,\,z]}$, $z \in \mathbb{R}$, in $|| F_{\!A} - e_{1,A} ||$ by functions $q_{z}$, $z \in \mathbb{R}$,
that are bounded and differen-
\linebreak
tiable. We define for $z \in \mathbb{R}$,
\[
\displaystyle{q_{z}(x) =
\left\{
\begin{array}{ll@{}}
1& \hspace*{2ex}
\text{for}\ x \leq z,\\[1.8ex]
1 - \dfrac{1}{2}\,\bigl( (x - z)/D_{\!A} \bigr)^2& \hspace*{2ex}
\text{for}\ z \leq x \leq z + D_{\!A},\\[2.3ex]
\dfrac{1}{2}\,\bigl((z + 2 D_{\!A} - x)/D_{\!A}\bigr)^2& \hspace*{2ex}
\text{for}\ z + D_{\!A} \leq x \leq z + 2 D_{\!A},\\[2.8ex]
0 & \hspace*{2ex}
\text{for}\ z + 2 D_{\!A} \leq x\,.
\end{array}  \right.}
\]
Note that we have
\begin{equation} \label{eq5-1}
q_{z}'(x + y) - q_{z}'(x) = y \biggl( \dfrac{1}{D_{\!A}^{2}} \int_{0}^{1} 
\Bigl( 1_{(z + D_{\!A},\,z + 2D_{\!A}]} - 1_{(z,\,z + D_{\!A}]} \Bigr) (x + sy)\,ds \biggr)
\end{equation}
for all $x, y, z \in \mathbb{R}$. In this section (\ref{eq5-1}) together with (\ref{eq2-5}) will play the
role of 
\linebreak
$\big| q'(x+y) - q'(x) \big| \leq |y|\, ||q''||$ of the last section.

The further use of $q_{z}$ instead of $1_{(-\infty,\,z]}$ will be justified in the following lemma.
This lemma needs (\ref{eq2-5}) (but only for $B = \hat{A}$) for the first time.
\begin{lemma}\label{lem5-2}
\[
\big|\big| F_{\!A} - e_{1,A} \big|\big| \leq \sup\limits_{z \in \mathbb{R}}
\bigg| E\bigl(q_{z}(T_{\!A})\bigr) - \int\nolimits_{\mathbb{R}} q_{z} e'_{1,A}\, dx \bigg| + \bigl( C_{1} + 1 \bigr) D_{\!A}^{2}.
\]
\end{lemma}

\begin{proof}
We use the abbreviation $D = D_{\!A}$. Then we have
\[
\begin{array}{l@{\hspace*{0.8ex}}c@{\hspace*{0.8ex}}l}
\Big| F_{\!A}(z) - E\bigl(q_{z-D}(T_{\!A})\bigr) \Big|&=
&\displaystyle{\bigg| \int\nolimits_{(z - D,\, z]} \dfrac{1}{2} 
\biggl( \dfrac{x - (z - D)}{D} \biggr)^{2}\,F_{\!A}(dx)}\\[3.5ex]
&&-\ \displaystyle{\int\nolimits_{(z,\, z + D]} \dfrac{1}{2} 
\biggl( \dfrac{x - (z + D)}{D} \biggr)^{2}\,F_{\!A}(dx) \bigg|}.
\end{array}
\]
Partial integration of each integral gives
\[
\begin{array}{l@{\hspace*{0.8ex}}c@{\hspace*{0.8ex}}l}
\hspace*{21.4ex}&=
&\displaystyle{\bigg| F_{\!A}(z) - \int\nolimits_{z-D}^{z} \dfrac{x - (z - D)}{D^2} F_{\!A}(x)\,dx}\\[3.5ex]
&&+\ \displaystyle{\int\nolimits_{z}^{z+D} \dfrac{x - (z + D)}{D^2} F_{\!A}(x)\,dx \bigg|}.
\end{array}
\]
Substitution of $y = z - x$ ($y = x - z$) in the first (second) integral leads to
\begin{align*}
\hspace*{21.4ex}&=
\displaystyle{\bigg| \int\nolimits_{0}^{D} \dfrac{D - y}{D^2} \Delta_{y}^{2} F_{\!A}(z - y)\,dy \bigg|}\\[2.5ex]  
&\leq
\displaystyle{\int\nolimits_{0}^{D} \dfrac{D - y}{D^2} C_{1}\bigl( D^{2} + y^{2} \bigr)\,dy 
\leq C_{1} D^{2}}.
\end{align*}
Similar computations give
\begin{equation*}
\displaystyle{\bigg| e_{1,A}(z) - \int\nolimits_{\mathbb{R}} q_{z - D}\, e'_{1,A}\, dx \bigg| =
\bigg| \int\nolimits_{0}^{D} \dfrac{D - y}{D^2} \Delta_{y}^{2} e_{1,A}(z - y)\,dy \bigg|}.
\end{equation*}
Now, using $\big|\big|\Delta_{y}^{2} e_{1,A}\big|\big| \leq y^2 \big|\big|e_{1,A}''\big|\big| \leq y^2 (1 + \beta_{\!A}/n)$
and $\beta_{\!A}/n \leq \epsilon_{0} \leq 1$, we obtain the lemma.
\end{proof}

The rest of the proof of Theorem \ref{th2-4} has many parallels to that of Theroem \ref{th2-1}.
\begin{lemma}\label{lem5-3}
There exist positive constants $c_{1}$ and $c_{2}$ such that
\[
\sup\limits_{z \in \mathbb{R}} |R(q_{z})| \leq \bigl( c_{1} C_{1} + c_{2} \bigr) D^{2}_{\!A}.
\]
\end{lemma}

\begin{proof}
We fix $z \in \mathbb{R}$. Therefore, for simplicity we drop the index $z$ of $q_{z}$ and $f_{z}$
and the index $A$ of $D_{\!A}$.

We adopt the proof of Lemma \ref{lem4-1} up to (\ref{eq4-4}) (with
$c_{\textbf{.}}||q|| + c_{\textbf{.}}||q'|| + c_{\textbf{.}}||q''||$ 
replaced by $c_{\textbf{.}}C_{1} + c_{\textbf{.}}$ in (\ref{eq4-2}), (\ref{eq4-3}) and (\ref{eq4-4})).
To prove the analogue of (\ref{eq4-4}), we consider
\begin{align*}
\Big| E\Bigl(f'&'\bigl(\,T_{B} + a + p + tr\,\bigr) - f''\bigl(\,T_{B} + a\,\bigr) \Bigr) \Big|\\[1.7ex]
\leq E\Bigl( \Big| (f&''-q')\bigl(\,T_{B} + a + p + tr\,\bigr) - (f''-q')\bigl(\,T_{B} + a\,\bigr) \Big| \Bigr)\\[1.7ex] 
&+ \Big| E\Bigl( q'\bigl(\,T_{B} + a + p + tr\,\bigr) - q'\bigl(\,T_{B} + a\,\bigr) \Bigr) \Big|.
\end{align*}
Now we use estimates similar to (\ref{eq4-5})$-$(\ref{eq4-8}) and then
$|q'(w)| \leq (1/D)1_{(z,\,z+2D]}(w)$ for all $w \in \mathbb{R}$ for the first summand, and
(\ref{eq5-1}) for the second summand. We get
\begin{align*}
\leq 12\,\bigl( |p| &+ |r| \bigr) \biggl\{ 1 + E\bigl( (T_{B} + a)^2  \bigr)\\[1.7ex]
&+ \displaystyle{\dfrac{1}{D} \int\nolimits_{0}^{1}
E\Bigl( \big|T_{B} + a\big|\, 1_{(z,\,z + 2D]}\bigl(\,T_{B} + a + sp + str\,\bigr)\Bigr)\,ds}\\[1.7ex]
&+ \displaystyle{\dfrac{1}{D^2} \int\nolimits_{0}^{1}
\Big|\,E\Bigl( \bigl(1_{(z + D,\,z + 2D]} - 1_{(z,\,z + D]}\bigr)
\bigl(\,T_{B} + a + sp + str\,\bigr)\Bigr)\,\Big|\,ds}\,\biggr\}\\[1.7ex]
= 12\,\bigl( |p| &+ |r| \bigr)\,\bigl\{ 1 + A_{1} + A_{2} + A_{3} \bigr\}.
\end{align*}
$A_{1}$ is estimated in (\ref{eq4-10}), so that it remains to estimate $A_{2}$ and $A_{3}$.

For $A_{2}$ we use (\ref{eq1-1}), the following Proposition \ref{prop5-7}, $|a| \leq 8$ and (\ref{eq4-9}):
\begin{equation} \label{eq5-4}
\begin{array}{@{}l@{\hspace*{0.1ex}}r@{\hspace*{0.3ex}}l@{}}
E\Bigl( \big|&T_{B}& +\, a\big|\, 1_{(z,\,z + 2D]}\bigl(\,T_{B} + a + sp + str\,\bigr)\Bigr)\\[2ex]
&=&\hspace*{0.2ex}E\Bigl( \big|\sigma_{\!B} \mathscr{T}_{B} + \mu_{B} + a\big|\, 1_{(z,\,z + 2D]}
\bigl(\,\sigma_{\!B} \mathscr{T}_{B} + \mu_{B} + a + sp + str\,\bigr)\Bigr)\\[2ex]
&\leq&c_{3}\Bigl( \bigl( \beta_{B}/(n-l)\bigr) + D \Bigr)\\[2.7ex]
&\leq&c_{4} D.
\end{array}
\end{equation}
This yields $A_{2} \leq c_{4}$. Finally, (\ref{eq2-5}) gives
\[
\Big|\,E\Bigl( \bigl(1_{(z + D,\,z + 2D]} - 1_{(z,\,z + D]}\bigr)
\bigl(\,T_{B} + a + sp + str\,\bigr)\Bigr)\,\Big|
\leq \big|\big|\Delta_{D}^{2} F_{B}\big|\big| \leq 2 C_{1} D^{2}.
\]
and so we have $A_{3} \leq 2C_{1}$. This proves the lemma.
\end{proof}

In order to complete the proof of Theorem \ref{th2-4} we need
\begin{lemma}\label{lem5-5}
There exists a positive constant $c$ such that
\[
\displaystyle{\sup\limits_{z \in \mathbb{R}} 
\bigg|\, \dfrac{1}{2} E\bigl( T_{\!A}^{3} \bigr) E\bigl( T_{\!A} f_{z}'(T_{\!A}) \bigr) -
\dfrac{1}{2} \lambda_{1,A} \Phi\bigl( x f_{z}'(x) \bigr) \bigg|
\leq c D^{2}_{\!A}.
}
\]
\end{lemma}

\begin{proof}
If we proceed as in the proof of Lemma \ref{lem4-11}, we see that it remains
to prove the following analogue of (\ref{eq4-12}): \phantom\qedhere  
\begin{equation} \label{eq5-6}
\sup\limits_{z \in \mathbb{R}}
\Big| E\bigl( T_{\!A} f_{z}'(T_{\!A}) \bigr) - \Phi\bigl( x f_{z}'(x) \bigr) \Big|
\leq c_{1} \beta_{\!A}/n.
\end{equation}
Using (\ref{eq3-2}) and the fact that the functions $q_{z}(x)$ and $\Phi(q_{z}) - x f_{z}(x)$ are
monotone decreasing between 0 and 1 [cf. \citet{Schneller2025ThesisV2}, Lemma 2.1.13(a)], 
we can deduce (\ref{eq5-6}) from the following general
result which is interesting in its own right and which was already needed for proving (\ref{eq5-4}).
\end{proof}

\begin{proposition}\label{prop5-7}
There exists a positive constant $c$ such that for all $A$ with $\sigma_{\!A} > 0$,
\[
\displaystyle{\sup\limits_{z \in \mathbb{R}} 
\Big|\, E\Bigl( |\mathscr{T}_{\!A}| 1_{(-\infty,\,z]}(\mathscr{T}_{\!A})\Bigr)
- \Phi\Bigl( |x| 1_{(-\infty,\,z]}(x)\Bigr)\,\Big|
\leq c \beta_{\!A}/n}.
\]
\end{proposition}

\begin{proof}[Sketch of Proof]
Proceed as in \citet{Bolthausen1984} with the quantities
\begin{align*}
h(x)&= h_{z,\lambda}(x) = |x|\Bigl\{ \bigl( (1 + (z - x)/\lambda) \wedge 1 \bigr) \vee 0 \Bigr\},\\[1.5ex] 
f(x)&= f_{z,\lambda}(x) = \bigl( \Theta h_{z,\lambda} \bigr) (x)
\end{align*}
and note that Bolthausen's $T_{\!A}$ is our $\mathscr{T}_{\!A}$.

We make two remarks. First, as in the proof of (\ref{eq4-12}) we find positive constants $c_{1}$, $c_{2}$, $c_{3}$
with $||f|| \leq c_{1}$ and $|f'(x)| \leq c_{2}|x| + c_{3}$ for all $x \in \mathbb{R}$. Instead of
Bolthausen's (2.5) we, therefore, have to use
\[
\displaystyle{\big| f'(x+y) - f'(x) \big| \leq |y| \bigg\{\,c_{4}(1 + x^2)(1 + |y|) + \dfrac{|x|}{\lambda}
\int\nolimits_{0}^{1} 1_{(z,\, z + \lambda ]}(x + sy)\,ds\,\bigg\}}.
\]
Our second remark concerns Bolthausen's (3.6). In order to prove our analogue
we define $a_{ij}' = \hat{a}_{ij} 1_{\{|\hat{a}_{ij}|\,\leq\, 1/2\}}(\hat{a}_{ij})$, 
$\overline{a}_{ij} = \widehat{(a')}_{ij}$ and $\Gamma = \{(i,j): |\hat{a}_{ij}|  > 1/2 \}$.\\ 
\mbox{\rule[0ex]{0ex}{2.8ex}We need the estimate}
\begin{equation} \label{eq5-8}
\displaystyle{\Big|\, E\Bigl( |\mathscr{T}_{\!A}| 1_{(-\infty,\,z]}(\mathscr{T}_{\!A})\Bigr)
- E\Bigl( |T_{\overbar{A}}| 1_{(-\infty,\,z]}(T_{\!A'})\Bigr)\,\Big|
\leq c_{5} \beta_{\!A}/n}.
\end{equation}
From this we obtain the following analogue of Bolthausen's (3.6):
\begin{align*}
\displaystyle{\sup\limits_{z} 
\Big|}&E\displaystyle{\Bigl( |\mathscr{T}_{\!A}| 1_{(-\infty,\,z]}(\mathscr{T}_{\!A})\Bigr)
- \Phi\Bigl( |x| 1_{(-\infty,\,z]}(x)\Bigr)\,\Big|}\\[1.5ex]
&\leq \displaystyle{\sup\limits_{z} \Big|\,E\Bigl( |T_{\overbar{A}}| 1_{(-\infty,\,z]}(T_{\!A'})\Bigr)
- \Phi\Bigl( |x| 1_{(-\infty,\,z]}(x)\Bigr)\,\Big| + c_{5} \beta_{\!A}/n}\\[1.5ex]
&\leq \displaystyle{\sup\limits_{z} 
\Big|\,E\Bigl( |T_{\overbar{A}}| 1_{(-\infty,\,z]}(T_{\overbar{A}})\Bigr)
- \Phi\Bigl( |x| 1_{(-\infty,\,z]}(x)\Bigr)\,\Big|}\\[1ex]
&\quad + \displaystyle{\sup\limits_{z}
\Big|\,\Phi\Bigl( |x| 1_{(-\infty,\,(z - \mu_{A'})/\sigma_{\!A'}]}(x)\Bigr)
- \Phi\Bigl( |x| 1_{(-\infty,\,z]}(x)\Bigr)\,\Big|  + c_{5} \beta_{\!A}/n}.
\end{align*}
Clearly, the second summand is of order $\beta_{\!A}/n$ [cf. Bolthausen's (3.2) and (3.3)] and
therefore our analogue is established with the exception of (\ref{eq5-8}):
\begin{align*}
\Big|\, E\Bigl( |\mathscr{T}_{\!A}|&1_{(-\infty,\,z]}(\mathscr{T}_{\!A})\Bigr)
- E\Bigl( |T_{\overbar{A}}| 1_{(-\infty,\,z]}(T_{\!A'})\Bigr)\,\Big|\\[2ex]
\leq\,\, &\Big|\, E\Bigl( |\mathscr{T}_{\!A}| 1_{(-\infty,\,z]}(\mathscr{T}_{\!A})\Bigr)
- E\Bigl( |T_{\overbar{A}}| 1_{(-\infty,\,z]}(\mathscr{T}_{\!A})\Bigr)\,\Big|\\[2ex]
&+ \Big|\, E\Bigl( |T_{\overbar{A}}| 1_{(-\infty,\,z]}(\mathscr{T}_{\!A})\Bigr)
- E\Bigl( |T_{\overbar{A}}| 1_{(-\infty,\,z]}(T_{\!A'})\Bigr)\,\Big|\\[2ex]
\leq\,\, &E\Bigl( \big|\mathscr{T}_{\!A} - T_{\overbar{A}}\big| \Bigr)
+ E\Bigl( |T_{\overbar{A}}| 1_{\{\mathscr{T}_{\!A} \not= T_{\!A'}\}} \Bigr) = B_{1} + B_{2}.
\end{align*}
Furthermore, we have
\begin{align*}
B_{1} &\leq \displaystyle{E\Bigl( \big|\mathscr{T}_{\!A} - T_{\!A'}\big| \Bigr)
+ E\Bigl( \big| \sigma_{\!A'}T_{\overbar{A}} + \mu_{A'} - T_{\overbar{A}}\big| \Bigr)}\\[1.5ex]
&\leq \displaystyle{\sum\limits_{i} E\Bigl( \big| \hat{a}_{i\pi(i)} - a_{i\pi(i)}' \big| \Bigr) 
+ c_{6}\beta_{\!A}/n}\\[1.2ex]
&= \displaystyle{(1/n) \sum\limits_{(i,j) \in \Gamma} |\hat{a}_{ij}| + c_{6}\beta_{\!A}/n
\leq c_{7}\beta_{\!A}/n},\\[1.7ex]
B_{2} &\leq \displaystyle{\sum\limits_{i} E\Bigl( |T_{\overbar{A}}|1_{\Gamma}\bigl(i,\pi(i)\bigr)\Bigr)}\\[1ex]
&= \displaystyle{\sum\limits_{i,j} (1/n)1_{\Gamma}(i,j) E\Bigl( |T_{\overbar{A}}|\, \big|\, \pi(i) = j \Bigr)}.
\end{align*}
If we argue as in \citet{Bolthausen1984}, page 385, lines 3$-$14 (cf. also our proof of Lemma \ref{lem4-1}),
we find a constant $c_{8}$ with $E\Bigl( |T_{\overbar{A}}|\, \big|\, \pi(i) = j \Bigr) \leq c_{8}$
for all $i$, $j$. Therefore, we have 
$B_{2} \leq c_{8}|\Gamma|/n \leq 8c_{8}\beta_{\!A}/n$.
\end{proof}

\begin{rem57} \label{rem5-9}
In \citet{Schneller2025ThesisV2}, Chapter 3, Section 6, a complete, but a bit different proof of 
Proposition \ref{prop5-7} may be found for the case $|\hat{a}_{ij}| \leq 1$.
\end{rem57}

\section{Proof of Theorem \ref{th2-12}(a)} \label{sec6}
In order to prove Theorem \ref{th2-12}(a) we use a result of \citet{vanZwet1982}, Theorem 2.1, which gives an
estimate of the characteristic function of $\mathscr{T}_{\!A}$ under the conditions (\ref{eq2-13})$-$(\ref{eq2-15}).
From this we will deduce (\ref{eq2-5}) with $C_{1} \approx (\log n)^{2}$.

We need some preliminaries. Let $U$ be a distribution function on $\mathbb{R}$ which has a density $u$ that is 
infinitely differentiable and has a support which is contained in $[-1,1]$. It follows 
[cf. \citet{feller1971introduction}, Chapter 15, Section 4, Lemma 4] that
\begin{equation} \label{eq6-1}
\displaystyle{|\hat{U}(t)| = {\scriptstyle \bigO}(|t|^{-n})\ \ \ \text{for}\,\, |t| \rightarrow \infty\,\,
\text{and}\,\, \text{all}\,\, n \in \mathbb{N}},
\end{equation}
where we denote by $\hat{G}$ the characteristic function of a distribution with distribution function $G$.

From (\ref{eq6-1}) we conclude
\begin{equation} \label{eq6-2}
\displaystyle{\int\nolimits_{\mathbb{R}} |t|^{n} |\hat{U}(t)|\,dt < \infty,\ \ \ \ \
\int\nolimits_{\mathbb{R}} |\hat{U}(t)|^{n}\,dt < \infty\ \ \ \text{for}\,\, \text{all}\,\, n \in \mathbb{N}}.
\end{equation}

Using this, we can prove
\begin{lemma} \label{lem6-3}
Let $F$ be a distribution function on $\mathbb{R}$ and
\[
\displaystyle{U_{\theta}(x) = U(x/\theta)\ \ \ \text{for}\,\, x \in \mathbb{R}\,\, \text{and}\,\, \theta > 0}.
\]
Then
\begin{equation} \label{eq6-4}
\displaystyle{\big|\Delta_{y}^{2} F(z)\big| \leq (y^2 + \theta) \int\nolimits_{\mathbb{R}}
(1 + |t|) \big| \hat{F}(t) \hat{U}_{\theta}(t) \big|\,dt}
\end{equation}
for all $z \in \mathbb{R}$, $y > 0$ and $\theta > 0$.
\end{lemma}

\begin{proof}
We use a technique introduced by \citet{vonBahr1967}, Section 3. Let $F_{\theta}$ be the convolution of 
$F$ and $U_{\theta}$. Then $F_{\theta}$ has a density $f_{\theta}$ and
\[
F_{\theta}(x - \theta) \leq F(x) \leq F_{\theta}(x + \theta)\ \ \ \text{for}\,\, \text{all}\,\,
x \in \mathbb{R},\,\, \theta > 0.
\]
Using this and the Plancherel identity, we obtain for $y \geq 2\theta$,
\begin{align*}
F( z &+ 2y ) - 2F( x + y ) + F(z)\\[2ex]
&\leq F_{\theta}( z + 2y + \theta) - 2 F_{\theta}( z + y - \theta) + F_{\theta}( z + \theta)\\[1.5ex]
&\leq \displaystyle{\bigg| \int\nolimits_{\mathbb{R}}
\Bigl( 1_{(z + y - \theta,\, z + 2y + \theta ]}(t) - 
1_{(z + \theta,\, z + y - \theta ]}(t) \Bigr) f_{\theta}(t)\,dt \bigg|}\\[2.5ex]
&= \displaystyle{\dfrac{1}{2 \pi}\, \bigg| \int\nolimits_{\mathbb{R}}
\biggl( \int\nolimits_{z + y - \theta}^{z + 2y + \theta} e^{- its}\,ds -
\int\nolimits_{z + \theta}^{z + y - \theta} e^{- its}\,ds \biggr) \hat{F}_{\theta}(t)\,dt \bigg|}\\[3ex]
&\leq \displaystyle{\dfrac{1}{2 \pi}\, \int\nolimits_{\mathbb{R}}
\dfrac{1}{|t|} \Big|\, e^{- it(z + 2y)} - 2 e^{- it(z + y)} + e^{- itz}\,\Big|\,\big|\hat{F}_{\theta}(t)\big|\,dt}\\[2.5ex]
&\quad + \displaystyle{\theta \int\nolimits_{\mathbb{R}} \big|\hat{F}_{\theta}(t)\big|\,dt}\\[2.5ex]
&\leq \displaystyle{y^{2}\! \int\nolimits_{\mathbb{R}} |t| \big|\hat{F}_{\theta}(t)\big|\,dt
+ \theta \int\nolimits_{\mathbb{R}} \big|\hat{F}_{\theta}(t)\big|\,dt}\\[2.5ex]
&\leq \displaystyle{(y^2 + \theta) \int\nolimits_{\mathbb{R}}
(1 + |t|) \big| \hat{F}(t) \hat{U}_{\theta}(t) \big|\,dt = :R}.
\end{align*}
For the last inequality we used $\hat{F}_{\theta} = \hat{F} \hat{U}_{\theta}$. In the case $y < 2\theta$
we proceed similarly. Furthermore, we obtain $-\bigl( F( z + 2y ) - 2F( x + y ) + F(z) \bigr) \leq R$
by completely analogous arguments.
\end{proof}

The assumption (\ref{eq2-5}) contains conditions for $B \in N(8,\hat{A})$. For this reason we need the estimate of van Zwet not only for $\hat{\mathscr{F}}_{A}$ but also for $\hat{F}_{B}$, $B \in N(8,\hat{A})$.
\begin{lemma} \label{lem6-5}
Let $A = (e_{i}d_{j})$ be an $n\,{\times}\,n$-matrix fulfilling the conditions 
(\ref{eq2-13})$-$(\ref{eq2-15})
and 
\[
n \geq n_{1} = \max\{n_{0},\,10,\,32/\delta\},\ \ \beta_{\!A} \leq \epsilon_{0} n,
\] 
where $n_{0}$ and $\epsilon_{0}$ are chosen such that $\frac{1}{2} \leq \sigma_{\!B}^2 \leq \frac{3}{2}$ 
holds for $B \in N(8,\hat{A})$ [cf. (\ref{eq4-9})].
Then there exist positive constants $b_{1}$, $b_{2}$, $b_{3}$ and $b_{4}$ depending only on
$e$, $E$, $d$, $D$, $\delta$ and $r$, $k$, $m$, $s$ such that
\begin{equation} \label{eq6-6}
\big| \hat{F}_{B}(t) \big| \leq b_{1} n^{- b_{2}\log n}\ \ \ \ 
\text{for}\,\,\, b_{3}\log n \leq |t| \leq b_{4}n^{3/2}\,\, \text{and}\,\,\, B \in N(8,\hat{A}).
\end{equation}
\end{lemma}

\begin{proof}
Use Theorem 2.1 of \citet{vanZwet1982} and proceed as in the 
proof of \citet{Schneller2025ThesisV2}, Proposition 4.3.9, pages 160$-$164. 

We remark that \citet{Bolthausen1984} uses, in the proof of our (\ref{eq4-9}), essentially $|\hat{a}_{ij}| \leq 1$.
If one uses $\sum_{i} \hat{a}_{ij} = 0 = \sum_{j} \hat{a}_{ij}$ and $|\hat{a}_{ij}| \leq 1 + |\hat{a}_{ij}|^{3}$
instead of $|\hat{a}_{ij}| \leq 1$, one can show our (\ref{eq4-9}) as well
[for more details see \citet{Schneller2025ThesisV2}, Chapter 3, Section 3]. 
\end{proof}
Our last lemma uses the estimate of (\ref{eq6-6}) in order to estimate the right-hand side of (\ref{eq6-4}).
\begin{lemma} \label{lem6-7}
Let $A = (e_{i}d_{j})$ be as in Lemma \ref{lem6-5}. Then there exists a positive constant $c$ depending only on
$e$, $E$, $d$, $D$, $\delta$ and $r$, $k$, $m$, $s$ such that
\begin{equation} \label{eq6-8}
\displaystyle{\int\nolimits_{\mathbb{R}}
(1 + |t|) \big| \hat{F}_{B}(t) \hat{U}_{D_{\!A}^{2}}(t) \big|\,dt \leq c(\log n)^{2}\ \ \ \ 
\text{for}\,\, \text{all}\,\,\, B \in N(8,\hat{A})}.
\end{equation}
\end{lemma}

\begin{proof}
It sufficies to prove (\ref{eq6-8}) with $|t|$ instead of $1 + |t|$. Using the abbre-
\linebreak
viation $\theta = D_{\!A}^{2}$
and the constants $b_{1}$, $b_{2}$, $b_{3}$ and $b_{4}$ from Lemma \ref{lem6-5} we have
\begin{align*}
\displaystyle{\int\nolimits_{\mathbb{R}} |t|} &\big| \hat{F}_{B}(t) \hat{U}_{\theta}(t) \big|\,dt\\[1.5ex]
&\leq \displaystyle{2\, \biggl\{
\int\nolimits_{0}^{b_{3} \log n} t\,dt +
\int\nolimits_{b_{3} \log n}^{b_{4} n^{3/2}} t\big|\hat{F}_{B}(t)\big|\,dt +
\int\nolimits_{b_{4} n^{3/2}}^{\infty} t\big|\hat{U}_{\theta}(t)\big|\,dt \biggr\}}\\[2ex]
&= 2\, \{ I_{1} + I_{2} + I_{3} \}.
\end{align*}
From Lemma \ref{lem6-5} we obtain $I_{2} \rightarrow 0$ as $n \rightarrow \infty$ 
so that $I_{1} + I_{2} \leq c_{1}(\log n)^{2}$. 
Furthermore, using $\theta \geq 1/(4n)$ and (\ref{eq6-2}) yields
\begin{align*}
I_{3} &= \displaystyle{\dfrac{1}{\theta^2} \int\nolimits_{b_{4} n^{3/2} \theta}^{\infty} v\big|\hat{U}(v)\big|\,dv
\ \ \ \ \ (v = t\theta)}\\[2ex]
&\leq \displaystyle{16n^2 \int\nolimits_{c_{2} n^{1/2}}^{\infty} v\big|\hat{U}(v)\big|\,dv
\ \ \ \ \ \Bigl( c_{2} = \dfrac{b_{4}}{4} \Bigr)}\\[2ex]
&\leq \displaystyle{16n^2 \int\nolimits_{c_{2} n^{1/2}}^{\infty}
\dfrac{v^{6}}{c_{2}^{5} n^{5/2}}\,\big|\hat{U}(v)\big|\,dv}\\[2ex]
&\leq \displaystyle{16 c_{2}^{-5} n^{-1/2} \int\nolimits_{\mathbb{R}} v^{6} \big|\hat{U}(v)\big|\,dv
\rightarrow 0\ \ \ \ \ \text{as}\,\, n \rightarrow \infty.}\qedhere
\end{align*}
\end{proof}
For $n\,{\times}\,n$-matrices $A$ with $n \geq n_{1}$ and $\beta_{\!A} \leq \epsilon_{0}n$ the estimate
(\ref{eq2-16}) now follows from the Lemmas \ref{lem6-3} and \ref{lem6-7}, from Theorem \ref{th2-4} and
\begin{equation*}
D_{\!A}^{2} \leq c n^{- 1\, +\, ((4/k)\, -\, 1)^{+}\,+\, ((4/s)\, -\, 1)^{+}}, 
\end{equation*}
where $c$ depends only on
$e$, $E$, $d$, $D$, $\delta$ and $r$, $k$, $m$, $s$. For the other matrices we obtain (\ref{eq2-16}), if
we choose $\mathscr{K}_{1}$ large enough (see the beginning of Section \ref{sec4}).

\section{Some remarks on the proofs for the second order} \label{sec7}
The proof of Theorem \ref{th2-7} is completely given in \citet{Schneller2025ThesisV2}, Chapter 3, Section 8. 
This proof has a structure which is similar to that of Theorem \ref{th2-4}, but it is much more 
extensive than the proof of Theorem \ref{th2-4}. We mention some essential differences.
\renewcommand{\theenumi}{\arabic{enumi}}  
\renewcommand{\labelenumi}{\theenumi.}
\begin{enumerate}
\item
The set $M_{16}$ has to be defined as a subset of $N^{16}$ so that the maps 
$M_{16} \ni \underline{i} \rightarrow u(\underline{i})$, $t(\underline{i})$ and $v(\underline{i})$, with 
$v(\underline{i})$ as follows, are well defined. The permutation $v(\underline{i})$ leaves
the numbers outside $\{i_{1},\ldots,i_{16}\}$ fixed and maps $i_{1},\ldots,i_{8}$
according to
\[
\begin{array}{@{}l@{\hspace*{2.8ex}}l@{\hspace*{2.8ex}}l@{\hspace*{2.8ex}}l@{}}
i_{1} \rightarrow i_{13}&i_{3} \rightarrow i_{15}&i_{5} \rightarrow i_{9}&i_{7} \rightarrow i_{11}\\[1.5ex]
i_{2} \rightarrow i_{14}&i_{4} \rightarrow i_{16}&i_{6} \rightarrow i_{10}&i_{8} \rightarrow i_{12}.
\end{array}
\]
With these definitions we define analogously five random permutations $\pi_{1},\ldots,\pi_{5}$.
\item
In order to obtain the right expansion in the analogue of equation (\ref{eq3-12}) and for the estimation
of the corresponding remainder terms $R(q_{z})$, $z \in \mathbb{R}$, we need an analogue of Proposition \ref{prop5-7}
for $E\bigl( \mathscr{T}_{\!A}^{2} 1_{(-\infty,\,z]}(\mathscr{T}_{\!A}) \bigr)$ and an Edgeworth
expansion of first order for $E\bigl( \mathscr{T}_{\!A} 1_{(-\infty,\,z]}(\mathscr{T}_{\!A}) \bigr)$.
\item
The condition (\ref{eq2-8}) is used mainly to prove the analogue of Lemma \ref{lem5-2}.
For the rest of the proof one needs only the weaker condition [cf. \citet{BiRob1982}, Lemma]\\[2ex]
(2.8')
$\begin{array}{@{\hspace*{14.5ex}}c@{}}
\textit{there}\,\, \textit{exists}\,\, \textit{a}\,\, \textit{positive}\,\, \textit{constant}\,\, C_{2}'\,\,
\textit{such}\,\, \textit{that}\\[1.5ex]
\big| \Delta_{y}^{3} F_{B}(z) \big| \leq C_{2}'\bigl( E_{\!A}^{3} + y^{3} \bigr)\\[1.5ex]
\textit{for}\,\, \textit{all}\,\, z \in \mathbb{R},\, 0 \leq y \leq E_{\!A}\,\, \textit{and}\,\, B \in N(16, \hat{A}),
\end{array}$\\[1.5ex]
\text{where}
\[
\Delta_{y}^{3} F_{B}(z) = \Delta_{y}^{2} F_{B}(z + y) - \Delta_{y}^{2} F_{B}(z).
\]
\end{enumerate}
Our next remarks concern the proof of Theorem \ref{th2-12}(b). This proof is also very similar to that of 
Theorem \ref{th2-12}(a) and is completely given in \citet{Schneller2025ThesisV2}, Chapter 4, Sections 1 and 3.

The appearance of the factor $n^{\epsilon}$ [instead of e.g., $(\log n)^{3}$] has its reason in establishing 
an estimation analogous to that of the term $I_{3}$ in the proof of Lemma \ref{lem6-7}.
For that we need $b_{4}n^{3/2}\theta \geq c n^{\delta}$ with a $\delta > 0$. Therefore, we have 
to take $\theta = n^{\delta}E_{\!A}^{3} \geq n^{\delta}/(2^{5/2} n^{3/2})$ instead of $\theta = E_{\!A}^{3}$.
For technical reasons $\delta = \epsilon/2$ is selected in \citet{Schneller2025ThesisV2}
[cf. especially Lemma 4.1.7(b) and the proof of (4.3.20)].

\section*{Acknowledgments} \label{sec8}
This paper is based on a part ot the author's thesis written at the
Technische Universit\"at Berlin [\citet{Schneller1987}]. 
The author wishes to thank Professor Dr. E. Bolthausen for
suggesting the problem and for stimulating discussions. Furthermore, the author is indebted to the
Associate Editor and the referees for their various suggestions which improved the original
manuscript of \citet{10.1214/aos/1176347258}. 
Next, the author would like to thank his son David Schneller for finding a misleading formulation 
in \citet{Schneller2025ThesisV2} that would otherwise have been used in this paper as well.

Last but not least, the author would like to express his gratitude to his dear wife, Antonia Schneller, 
for mental support while writing this paper and \citet{Schneller2025ThesisV2} after his retirement.


\begin{thebibliography}{20}

\bibitem[Barbour(1986)]{Barbour1986}
\textsc{Barbour, A. D.} (1986). 
Asymptotic expansions based on smooth functions in the central limit theorem.
\textit{Probab. Theory Related Fields}
\textbf{72} 289--303.
DOI: \href{https://doi.org/10.1007/BF00699108}{10.1007/BF00699108}.

\bibitem[Bickel and Robinson(1982)]{BiRob1982}
\textsc{Bickel, P. J. and Robinson, J.} (1982). 
Edgeworth expansions and smoothness.
\textit{Ann. Probab.}
\textbf{10} 500--503.
DOI: \href{https://doi.org/10.1214/aop/1176993873}{10.1214/aop/1176993873}.

\bibitem[Bickel and van Zwet(1978)]{BivanZw1978}
\textsc{{Bickel}, P. J. and {van Zwet}, W. R.} (1978). 
Asymptotic expansions for the power of distribution-free tests in the two-sample problem.
\textit{Ann. Statist.}
\textbf{6} 937--1004.
DOI: \href{https://doi.org/10.1214/aos/1176344305}{10.1214/aos/1176344305}.

\bibitem[Bolthausen(1984)]{Bolthausen1984}
\textsc{Bolthausen, E.}  (1984). 
An estimate of the remainder in a combinatorial central limit theorem.
\textit{Z. Wahrsch. verw. Gebiete}
\textbf{66} 379--386.
DOI: \href{https://doi.org/10.1007/BF00533704}{10.1007/BF00533704}.

\bibitem[Does(1982)]{Does1982}
\textsc{Does, R. J. M. M.}  (1982). 
Berry-Esseen theorems for simple linear rank statistics under the null-hypothesis.
\textit{Ann. Probab.}
\textbf{10} 982--991.
DOI: \href{https://doi.org/10.1214/aop/1176993719}{10.1214/aop/1176993719}.

\bibitem[Does(1983)]{Does1983}
\textsc{Does, R. J. M. M.}  (1983). 
An Edgeworth expansion for simple linear rank statistics under the null-hypothesis.
\textit{Ann. Statist.}
\textbf{11} 607--624.
DOI: \href{https://doi.org/10.1214/aos/1176346166}{10.1214/aos/1176346166}.

\bibitem[Erickson(1974)]{Erickson1974}
\textsc{Erickson, R. V.}  (1974). 
$L_{1}$ bounds for asymptotic normality of $m$-dependent sums using Stein's technique.
\textit{Ann. Probab.}
\textbf{2} 522--529.
DOI: \href{https://doi.org/10.1214/aop/1176996670}{10.1214/aop/1176996670}.

\bibitem[Feller(1971)]{feller1971introduction}
\textsc{Feller, W.}  (1971). 
\textit{An Introduction to Probability Theory and Its Application \textbf{2}},
2nd ed. Wiley, New York.

\bibitem[H{\'a}jek et al.(1999)]{hajek+sidak+sen1999}
\textsc{H{\'a}jek, J. and \v{S}id{\'a}k, Z. and Sen, P. K.}  (1999). 
\textit{Theory of Rank Tests},
2nd ed. Academic Press, San Diego and London.
DOI: \href{https://doi.org/10.1016/B978-0-12-642350-1.X5017-6}{10.1016/B978-0-12-642350-1.X5017-6}.

\bibitem[Ho and Chen(1978)]{HoChen1978}
\textsc{{Ho}, S. T. and {Chen}, L. H. Y.}  (1978). 
An $L_{p}$ bound for the remainder in a combinatorial central limit theorem.
\textit{Ann. Probab.}
\textbf{6} 231--249.
DOI: \href{https://doi.org/10.1214/aop/1176995570}{10.1214/aop/1176995570}.

\bibitem[Hoeffding(1951)]{Hoeffding1951}
\textsc{Hoeffding, W.}  (1951). 
A combinatorial central limit theorem.
\textit{Ann. Math. Statist.}
\textbf{22} 558--566.
DOI: \href{https://doi.org/10.1214/aoms/1177729545}{10.1214/aoms/1177729545}.

\bibitem[Motoo(1957)]{Motoo1957}
\textsc{Motoo, M.}  (1957). 
On the Hoeffding's combinatorial central limit theorem.
\textit{Ann. Inst. Statist. Math.}
\textbf{8} 145--154.
DOI: \href{https://doi.org/10.1007/BF02863580}{10.1007/BF02863580}.

\bibitem[Robinson(1978)]{Rob1978}
\textsc{Robinson, J.}  (1978). 
An asymptotic expansion for samples from a finite population.
\textit{Ann. Statist.}
\textbf{6} 1005--1011.
DOI: \href{https://doi.org/10.1214/aos/1176344306}{10.1214/aos/1176344306}.

\bibitem[Schneller(1987)]{Schneller1987}
\textsc{Schneller, W.} (1987).
\textit{Edgeworth-Entwicklungen f{\"u}r lineare Rangstatistiken}.
Ph.D. thesis, Technische Universit{\"a}t Berlin.

\bibitem[Schneller(1988)]{Schneller1988}
\textsc{Schneller, W.} (1988).
A short proof of Motoo's combinatorial central limit theorem using Stein's method.
\textit{Probab. Theory Related Fields}
\textbf{78} 249--252.
DOI: \href{https://doi.org/10.1007/BF00322021}{10.1007/BF00322021}.

\bibitem[Schneller(1989)]{10.1214/aos/1176347258}
\textsc{Schneller, W.} (1989).
Edgeworth expansions for linear rank statistics.
\textit{Ann. Statist.}
\textbf{17} 1103--1123.
DOI: \href{https://doi.org/10.1214/aos/1176347258}{10.1214/aos/1176347258}.

\bibitem[Schneller(2025)]{Schneller2025ThesisV2}
\textsc{Schneller, W.} (2025).
\textit{Edgeworth Expansions for Linear Rank Statistics Using Stein's Method}
[English translation of an updated version of the author's Ph.D. thesis from 1987].
Available at \href{http://arxiv.org/abs/2511.12187}{arXiv:2511.12187}.

\bibitem[Schneller(2026)]{10.1214/25-AOS2611}
\textsc{Schneller, W.} (2026).
Erratum: Edgeworth expansions for linear rank statistics.
\textit{Ann. Statist.}
\textbf{54} 1649--1653.
DOI: \href{https://doi.org/10.1214/25-AOS2611}{10.1214/25-AOS2611}.

\bibitem[Schneller(2026a)]{10.1214/25-AOS2611SUPPA}
\textsc{Schneller, W.} (2026a).
Supplement A to {''Erratum: Edgeworth expansions for linear rank statistics.''}
DOI: \href{https://doi.org/10.1214/25-AOS2611SUPPA}{10.1214/25-AOS2611SUPPA}.

\bibitem[Schneller(2026b)]{10.1214/25-AOS2611SUPPB}
\textsc{Schneller, W.} (2026b).
Supplement B to {''Erratum: Edgeworth expansions for linear rank statistics.''}
DOI: \href{https://doi.org/10.1214/25-AOS2611SUPPB}{10.1214/25-AOS2611SUPPB}.

\bibitem[Stein(1972)]{stein1972bound}
\textsc{Stein, C.}  (1972). 
A bound for the error in the normal approximation to the distribution of a sum of dependent random variables.
\textit{Proc. Sixth Berkeley Symp. Math. Statist. Probab.}
\textbf{2} 583--602.
Univ. California Press
URL: \href{https://projecteuclid.org/ebooks/berkeley-symposium-on-mathematical-statistics-and-probability/Proceedings-of-the-Sixth-Berkeley-Symposium-on-Mathematical-Statistics-and/chapter/A-bound-for-the-error-in-the-normal-approximation-to/bsmsp/1200514239}
{https://projecteuclid.org/ebooks/berkeley-symposium-on-mathematical-statistics-and-probability/Proceedings-of-the-Sixth-Berkeley-Symposium-on-Mathematical-Statistics-and/chapter/A-bound-for-the-error-in-the-normal-approximation-to/bsmsp/1200514239}.

\bibitem[van Zwet(1982)]{vanZwet1982}
\textsc{{van Zwet}, W. R.}  (1982). 
On the Edgeworth expansion for the simple linear rank statistic.
In \textit{Colloquia Mathematica Societatis J{\'a}nos Bolyai} (B. V. Gnedenko, M. L. Puri and I. Vincze, eds.)
\textbf{32} 889--909.
North-Holland, Amsterdam.
URL [Preprint]: \href{https://ir.cwi.nl/pub/8045}{https://ir.cwi.nl/pub/8045}.

\bibitem[von Bahr(1967)]{vonBahr1967}
\textsc{{von Bahr}, B.}  (1967). 
Multi-dimensional integral limit theorems.
\textit{Ark. Mat.}
\textbf{7} 71--88.
DOI: \href{https://doi.org/10.1007/BF02591678}{10.1007/BF02591678}.

\end{thebibliography}
\end{document}